\documentclass{scrartcl}

\usepackage[utf8]{luainputenc}
\usepackage[UKenglish]{babel}
\usepackage{csquotes}

\usepackage[a4paper]{geometry}
\usepackage[plainpages=false,pdfpagelabels,hidelinks,unicode]{hyperref}

\usepackage[%
  backend=bibtex,bibencoding=ascii,
  style=numeric-comp,
  giveninits=true, uniquename=init, 
  natbib=true,
  url=true,
  doi=true,
  isbn=false,
  backref=false,
  maxnames=99,
  ]{biblatex}
\usepackage{amsmath}
\allowdisplaybreaks
\usepackage{amssymb}
\usepackage{commath}
\usepackage{mathtools}
\usepackage{bbm}
\usepackage{nicefrac}

\usepackage{amsthm}
\theoremstyle{plain}
  \newtheorem{theorem}{Theorem}
  
  \newtheorem{proposition}[theorem]{Proposition}
\theoremstyle{definition}
  
  \newtheorem{remark}[theorem]{Remark}
  \newtheorem{example}[theorem]{Example}
\numberwithin{theorem}{section}

\usepackage{color}
\usepackage{graphicx}
\usepackage[small]{caption}
\usepackage{subcaption}

\begingroup\expandafter\expandafter\expandafter\endgroup
\expandafter\ifx\csname pdfsuppresswarningpagegroup\endcsname\relax
\else
\fi

\usepackage{booktabs}
\usepackage{rotating}
\usepackage{multirow}

\usepackage{ifluatex}
\ifluatex
  \usepackage[no-math]{fontspec}
\else
  \usepackage[T1]{fontenc}
\fi
\usepackage{newpxtext,newpxmath}

\newcommand{\R}{\mathbb{R}}
\newcommand{\scp}[2]{\left\langle{#1,\, #2}\right\rangle}
\newcommand{\I}{\operatorname{I}}

\let\phi\varphi
\let\rho\varrho

\renewcommand{\r}{\mathfrak{r}}
\newcommand{\fnum}{f^{\mathrm{num}}}
\newcommand{\fnumj}{f^{\mathrm{num},j}}

\newcommand{\fvolj}{f^{\mathrm{vol},j}}

\newcommand{\Fnumj}{F^{\mathrm{num},j}}
\newcommand{\Snumj}{S^{\mathrm{num},j}}
\newcommand{\VOL}{\mathrm{VOL}}
\newcommand{\SURF}{\mathrm{SURF}}

\newcommand{\Ekin}{E_{\mathrm{kin}}}

\newcommand{\st}{\operatorname{s.t.}\;}

\newcommand{\Ifc}{\operatorname{I}_{f \leftarrow c}}
\newcommand{\Icf}{\operatorname{I}_{c \leftarrow f}}

\newsavebox{\DelimiterBox}
\newlength{\DelimiterHeight}
\newlength{\DelimiterDepth}
\newsavebox{\ArgumentBox}
\newlength{\ArgumentHeight}
\newlength{\ArgumentDepth}
\newlength{\ResizedDelimiterHeight}
\newlength{\ResizedDelimiterDepth}

\newcommand{\encloseby}[3]{%
  \savebox{\ArgumentBox}{$\displaystyle #1$}%
  \settoheight{\ArgumentHeight}{\usebox{\ArgumentBox}}%
  \settodepth{\ArgumentDepth}{\usebox{\ArgumentBox}}%
  \savebox{\DelimiterBox}{#2}%
  \settoheight{\DelimiterHeight}{\usebox{\DelimiterBox}}%
  \settodepth{\DelimiterDepth}{\usebox{\DelimiterBox}}%
  \setlength{\ResizedDelimiterHeight}{%
    \maxof{1.2\ArgumentHeight}{\DelimiterHeight}%
  }
  \setlength{\ResizedDelimiterDepth}{%
    \maxof{1.2\ArgumentDepth}{\DelimiterDepth}%
  }
  \raisebox{-\ResizedDelimiterDepth}{%
    \resizebox{\width}{\ResizedDelimiterHeight+\ResizedDelimiterDepth}{%
      \raisebox{\DelimiterDepth}{#2}%
    }%
  }
  #1
  \raisebox{-\ResizedDelimiterDepth}{%
    \resizebox{\width}{\ResizedDelimiterHeight+\ResizedDelimiterDepth}{%
      \raisebox{\DelimiterDepth}{#3}%
    }%
  }
}

\ifluatex
  \newcommand{\armean}[1]{\encloseby{#1}{$\{\mkern-5mu\{$}{$\}\mkern-5mu\}$}}
  \newcommand{\jump}[1]{\encloseby{#1}{$[\mkern-4mu[$}{$]\mkern-4mu]$}}
\else
  \newcommand{\armean}[1]{\encloseby{#1}{$\{\mkern-6mu\{$}{$\}\mkern-6mu\}$}}
  \newcommand{\jump}[1]{\encloseby{#1}{$[\mkern-3mu[$}{$]\mkern-3mu]$}}
\fi

\newenvironment{keywords}{\par\textbf{Key words.}}{\par}
\newenvironment{AMS}{\par\textbf{AMS subject classification.}}{\par}

\newcommand{\PP}{\mathbb P}

\newcommand{\VV}{\mathcal{V}}

\newcommand{\mean}[1]{\overline{#1}}

\renewcommand{\SS}{\mathcal{S}}
\newcommand{\ResKs}{{{\Phi}^K_\sigma}}
\newcommand{\tildResKs}{{{\tilde{\Phi}}^K_\sigma}}

\newcommand{\EE}{\mathcal{E}}

\title{Reinterpretation and Extension of Entropy Correction Terms for Residual
       Distribution and Discontinuous Galerkin Schemes: Application to Structure Preserving Discretization}
\author{Rémi Abgrall, Philipp Öffner, Hendrik Ranocha}
\date{January 4, 2022}

\makeatletter
\hypersetup{pdfauthor={\@author}}
\hypersetup{pdftitle={\@title}}
\makeatother

\begin{document}

\maketitle

\begin{abstract}
For the general class of residual distribution (RD) schemes, including many
finite element (such as continuous/discontinuous Galerkin) and
flux reconstruction methods, an approach to construct entropy conservative/ dissipative
semidiscretizations by adding suitable correction terms has been proposed
by Abgrall (J.~Comp.~Phys. 372: pp. 640--666, 2018).
In this work,
the correction terms are characterized as solutions of certain optimization
problems and are adapted to the SBP-SAT framework,
focusing on discontinuous Galerkin methods.
Novel generalizations to entropy inequalities, multiple
constraints, and kinetic energy preservation for the Euler equations are developed
and tested in numerical experiments. For all of these optimization problems,
explicit solutions are provided. Additionally, the correction approach is applied for the first time to obtain a fully discrete entropy conservative/dissipative  RD scheme. Here, the application of the deferred correction (DeC) method for the time integration is essential. This paper can be seen as describing a systematic method to construct structure preserving discretization, at least for the  considered example.
\end{abstract}

\begin{keywords}
  entropy stability,
  kinetic energy preservation,
  conservation laws,
  residual distribution schemes,
  discontinuous Galerkin schemes,
  Euler equations
\end{keywords}

\begin{AMS}
  65M12,  
  65M60,  
  65M70,  
  65M06   
\end{AMS}

\section{Introduction}

Consider a hyperbolic conservation law
\begin{equation}
\label{eq:CL}
  \partial_t u(t,x) + \sum_{j=1}^d \partial_j f^j\bigl( u(t,x) \bigr) = 0,
  \qquad
  t \in (0,T),\, x \in \Omega,
\end{equation}
in $d$ space dimensions such as the compressible Euler equations of gas dynamics,
where $u\colon (0,T) \times \Omega \to \Upsilon \subseteq \R^m$ are the conserved
variables, $f^j\colon \Upsilon \to \R^m$ the fluxes, and $t \in (0,T)$, $x \in \Omega
\subseteq \R^d$ the time and space coordinates, respectively. The
conservation law has to be equipped with appropriate initial and boundary conditions.

Given a convex entropy $U\colon \Upsilon \to \R$ with entropy variables
$w = \partial_u U$ and entropy fluxes $F^j\colon \Upsilon \to \R$ fulfilling
$\scp{w}{\partial_u f^j}=\partial_u F^j$,
smooth solutions of \eqref{eq:CL} satisfy $\partial_t U(u) + \sum_{j=1}^d
\partial_j F^j(u) = 0$ and the entropy inequality
\begin{equation}
\label{eq:entropy-inequality}
  \partial_t U(u) + \sum_{j=1}^d \partial_j F^j(u) \leq 0
\end{equation}
is used as admissibility criterion for weak solutions.
The mapping between the entropy variables $w = \partial_u U$ and the
conservative variables 
$u$ is one-to-one since $U$ is convex.

Since the seminal work of Tadmor~\cite{tadmor1987numerical,tadmor2003entropy},
there has been interest in techniques to mimic \eqref{eq:entropy-inequality}
for semidiscretizations of hyperbolic conservation laws. Some recent contributions
are, e.g.
\cite{carpenter2016towards, sjogreen2018high,chan2018discretely}.
Recently, relaxation Runge--Kutta methods have been proposed to transfer such
semidiscrete entropy conservation/dissipation  (SEC/D) results to fully discrete (FEC/D)
schemes \cite{ketcheson2019relaxation,ranocha2020relaxation,ranocha2020general,ranocha2020fully}.
Other possibilities are  to apply artificial viscosity or modal
filtering in an adaptive way
\cite{glaubitz2016artificial,offner2019analysis,sun2019enforcing}.

The EC/D semidiscretizations cited above
are built on the framework of EC numerical fluxes in
the sense of Tadmor
and their extension to
higher order methods \cite{ lefloch2002fully,
fernandez2014generalized, fisher2013high,
ranocha2018comparison, chen2017entropy}. However, these require special
quadrature rules in a finite element setting and cannot be applied to
all kinds of semidiscretizations and all kind of grids.
Additionally, their construction can
become rather complicated for complex physical models or even impossible, especially if
multiple secondary quantities are of interest, e.g. the entropy and the
kinetic energy for the Euler equations or further constraints such as
on the angular momentum.

In this article, the correction terms enforcing entropy conservation of
numerical methods in the general class of residual distribution (RD) schemes
proposed by Abgrall~\cite{abgrall2018general} and modifications suggested
in \cite{ranocha2020strong} are extended and analyzed. These schemes do not require
special quadrature rules nor grid structures and provide a general toolbox to enhance given
schemes with additional desired properties.

We characterize the entropy correction terms as solutions of optimization
problems, introducing different variants of this approach. Additionally,
new applications and generalizations are developed and compared in numerical
experiments with up to date methods.
For all of these optimization problems, explicit analytical
solutions are provided, resulting in reasonable schemes.

This article is structured as follows. Firstly, the numerical schemes and entropy
correction terms are introduced in Section~\ref{sec:correction-terms}, starting
with RD in Section~\ref{subsec:RD}. Here, already the extension to FEC/D using the correction term in the DeC-RD framework is explained.  Thereafter, discontinuous
element based schemes such as discontinuous Galerkin (DG) methods are described
in Section~\ref{subsec:DG} and the characterizations of entropy correction terms
as solutions of optimization problems are developed. Generalizations to
entropy inequalities, multiple linear constraints, and kinetic energy preservation
for the compressible Euler equations are developed in Section~\ref{sec:generalizations}.
Numerical examples using all these schemes are presented in
Section~\ref{sec:numerical-examples}. In Section~\ref{sec:remark}, 
we give some motivating examples why a new formulation to obtain EC/D
numerical schemes is useful.
We summarize and discuss our results in Section~\ref{sec:summary},
presenting also some directions of further research.
Additionally, we demonstrate how the correction terms can be used as a
procedure to grid refinement and coarsening in \autoref{sec:appendix}, yielding
EC/D grid transfer operations including numerical examples.

As indicated above, there are currently many entropy conservative and entropy
dissipative numerical fluxes; some are even kinetic energy compatible.
Thus, one may wonder why we develop a new solution. A literature check indicates
that all available solutions assume a calorically perfect gas. However, there are
many cases, such as combustion problems or multiphase flow, where the equation
of state is not that of a calorically perfect gas. In Section~\ref{sec:remark},
taking the example of one of these entropy conservative numerical fluxes,
we will point out that all works simply because of the special structure of the
flux in the calorically perfect gas case. If it is certainly possible, as this
has been done in the evaluation of the Roe average for non calorically perfect gas,
to workout such extensions, there will be cases where some ambiguity will still
exists (as it is the case for really nonlinear EOS, or tabulated ones), and in
any case, this will be a case-by-case analysis. If one wants to add additional
constraints, such as kinetic energy compatibility, or the local preservation
of kinetic momentum\footnote{The kinetic momentum also satisfies a conservation
law that is the consequence of the Euler equations.}, everything will need to
start from scratch, with additional constraints such as those proposed in
\cite{Jameson} for kinetic energy global conservation or \cite{ranocha2021preventing}
for pressure equilibria. The framework we propose in this paper completely avoids this.

\section{Entropy Corrections for Numerical Schemes}
\label{sec:correction-terms}

We will describe existing formulations of entropy correction terms for RD \&
DG schemes and present an interpretation in terms of a quadratic minimization
problem.

\subsection{Nodal Formulation: Residual Distribution Schemes}
\label{subsec:RD}

The first introduction of RD schemes
can be found in Roe's seminal work \cite{roe1981approximate}
and in the paper by Ni \cite{ni1981multiple}. Since then, further developments
have been done for generalization and to reach high order in the discretization, cf.\ \cite{abgrall2018general,abgrall2011construction} and references therein. The main advantage of the RD approach is the abstract formulation of
the schemes, working only with the degrees of freedom (DOFs).
The selection of approximation/solution space and the definition of the residuals
specifies the scheme completely and thus the properties of the considered methods.
Today, the RD ansatz provides a unifying framework including some -- if not most -- of
the up-to-date used high order methods like continuous/discontinuous Galerkin methods
and flux reconstruction schemes \cite{abgrall2018connection}.

\subsubsection*{Residual Distribution Schemes}

A classical time/space splitting using the method of lines will destroy
the order of accuracy of the RD approach. Hence,
the RD approach will be explained first for a steady state problem. After
introducing RD methods and the entropy corrections in this framework,
we will consider a temporal discretization using the deferred correction (DeC)
method following \cite{abgrall2017high, abgrall2017high2} for time-dependent
problems, including application of entropy corrections for these fully discrete
schemes. We will compare this DeC RD method with the variant of
\cite{ricchiuto2010explicit} using Runge--Kutta schemes in the numerical
experiments.

Consider the steady state problem
\begin{equation}
\label{eq:CL_steady}
 \sum_{j=1}^d \partial_j f^j\bigl( u(x) \bigr) = 0,
  \qquad
  x \in \Omega,
\end{equation}
of a hyperbolic conservation law \eqref{eq:CL} with suitable boundary conditions.
First, the domain $\Omega$ is split into subdomains $\Omega_l$ (e.g.\ simplex
or quad/hex elements in two/three dimensions).
$K$ denotes any generic element of the mesh and $h$ characterizes the  mesh size.
Boundary elements are denoted as $\Gamma$. Then, the DOFs $\sigma$  are defined  with respect to the splitting and the weights in
each $K$. For each $K$, the set of DOFs $\sum_K$ is given by linear forms acting on the set
$\PP^k$ of polynomials of degree $k$ such that the linear mapping
$q\in \PP^k\longmapsto (\sigma_1(q),\cdots, \sigma_{|\sum_K|}(q))$ is one-to-one.
$\SS$ denotes the set of DOFs in all elements.
The solution $u$ is approximated by an element of the space
\begin{equation}\label{eq:solution_space}
\VV^h:=\bigoplus_{K} \left\{ u^h \in L^2(K), u^h|_K \in \PP^k   \right\}.
\end{equation}
A linear combination of basis functions $\varphi_\sigma\in \VV^h$  is used to form the numerical solution
\begin{equation}\label{eq:solution_approx}
u^h(x)=\sum_{K\in \Omega_l }\sum_{\sigma \in K} u_{\sigma}^h  \varphi_{\sigma}|_K(x), \quad \forall{x \in \Omega},
\end{equation}
where the coefficients $u_{\sigma}^h$ must be found by a numerical method. Therefore, the residuals
come finally into play.
Now, the RD scheme can be formulated by the following three steps to
calculate the coefficients $u_{\sigma}^h$.
\begin{enumerate}
\item Define for any $K$ the total residual 
$\Phi^K$ of $\int_K \sum_{j=1}^d \partial_j f^j\bigl( u^h\bigr)$, i.e.
\begin{equation*}
\Phi^K =\oint_{\partial K} \sum_{j=1}^d f^j( u^h\bigr) \cdot \nu_j.
\end{equation*}
In the following, $\oint$ will be used to denote the discrete evaluation of
integrals by some quadrature rule.
Examples are given in \cite{abgrall2018general} and below.

\item Split the total residual into sub-residual $\ResKs$ for each degree of freedom $\sigma \in K$,
so that the sum of all the contributions over an element $K$ is the fluctuation term itself, i.e. 
for any element $K$
and any $u^h\in \VV^h$, we have
\begin{equation}\label{eq:residual_basic}
    \sum\limits_{\sigma \in K} \ResKs(u^h) =\oint_{\partial K}\sum\limits_j \fnumj \left(u_{|K}^h, u_{|K^-}^h \right) \cdot \nu_j,
\end{equation}
where $u_{|K}^h$ is the restriction of $u^h$ in the element $K$, $u_{|K^-}^h$
is the restriction of $u^h$ on the other side of the local edge/face of $K$,
and $\nu_j$ is the $j$-th component of the outer unit normal vector $\nu$ at $\partial K$.
In addition, $\fnum$ is a consistent numerical flux, i.e. $\fnumj(u,u)=f^j(u)$.

\item The resulting scheme is finally obtained by summing all sub-residuals of one degree of freedom from different elements $K$, i.e.
\begin{equation}\label{eq:total_scheme}
\sum_{K| \sigma \in K} \ResKs(u^h) = 0, \qquad \forall \sigma.
\end{equation}
The term
\eqref{eq:total_scheme} allows to calculate the coefficients $ u_{\sigma}^h$ in the  numerical
 approximation \eqref{eq:solution_approx} iteratively.
\end{enumerate}
If $\sigma \in \Gamma$, we write
\begin{equation}\label{eq:scheme}
 \sum_{K\in \Omega|\sigma \in K }\ResKs(u^h)+ \sum_{\Gamma \in \partial \Omega|\sigma \in \Gamma} \Phi^{\Gamma}_{\sigma}(u^h)=0,
\end{equation}
where $\Phi^{\Gamma}_{\sigma}$ denotes the boundary residual. They satisfy similar equations  as \eqref{eq:residual_basic}, see \cite{abgrall2018general} for details.
The RD scheme is described by \eqref{eq:total_scheme}--\eqref{eq:scheme}. This often results in a large system of nonlinear equations, which can be solved by an ad-hoc iterative method. To specify the method (FV, DG, etc.) completely, the  solution space \eqref{eq:solution_space}  (and its basis)  has to be chosen
and the exact definition of the  residuals $\ResKs$ has to be given.
For example, a  DG scheme in the RD framework is specified by choosing
the solution space  $\VV^h$ from \eqref{eq:solution_space},
the internal residuals
\begin{equation}\label{eq:Residual_DG}
  \ResKs(u^h) =-\oint_K \partial_j \varphi_\sigma \cdot f^j(u^h)+\oint_{\partial K } \varphi_\sigma  \fnumj (u^h_{|_K}, u^h_{|_{K^-}})\cdot \nu_j,
  \end{equation}
and the boundary residuals
 $ \Phi^{\Gamma}_{\sigma}(u^h) =\displaystyle\oint_\Gamma  \varphi_\sigma \bigl( \fnumj ( u^h_{|_\Gamma}, u_b) -f^j(u^h) \cdot \nu_j \bigr).$


\subsubsection*{Entropy Correction Term}

In \cite{abgrall2018general}, the author presented an approach to construct EC/D schemes in a general framework.
Therefore, a correction term is added to the scheme $\ResKs$ at every DOF $\sigma \in K$ to ensure that the scheme fulfills discretely the entropy condition \eqref{eq:entropy-inequality}. In terms of RD, an entropy conservative scheme%
\footnote{An entropy dissipative semidiscretization has an inequality in 
\eqref{eq:entropy-inequality}. In this part,  the steady state case \eqref{eq:CL_steady} is considered.}
fulfills 
\begin{equation}\label{eq:entropy_condition}
\sum_{\sigma\in K }\scp{w_\sigma}{\ResKs}=  \oint_{\partial K} \Fnumj  \left(w_{|K}^h, w_{|K^-}^h \right) \nu_j,
\end{equation}
where $\Fnumj$ is a numerical entropy flux and
$w_{\sigma}$ is the entropy variable at the DOF $\sigma$.
In general, the conservation relations \eqref{eq:residual_basic}
and \eqref{eq:entropy_condition} are not compatible. To achieve both,
we manipulate the residuals as follows.
The correction term $ r_\sigma^K$ is added to the  residual $\ResKs$ at every
degree of freedom, such that the corrected residual
\begin{equation}\label{eq:residual_corr}
\tildResKs= \ResKs+r_\sigma^K
\end{equation}
fulfills the discrete entropy condition \eqref{eq:entropy_condition}.
In \cite{abgrall2018general}, the following correction term is introduced (remember $w$ is the entropy variable)
\begin{gather}
  \label{eq:correction}
  r_\sigma^K = \alpha(w_\sigma -\mean{w} ),
  \quad \text{with }
  \mean{w} = \frac{1}{\#K}  \sum_{\sigma \in K } w_\sigma,
  \\
  \label{eq:error_abgrall}
  \alpha = \frac{\EE}{\sum_{\sigma \in K } \left(w_\sigma  -\mean{w} \right)^2},
  \quad
  \EE:= \oint_{\partial K} \Fnumj  \left(w_{|K}^h, w_{|K^-}^h \right) \nu_j -
  \sum_{\sigma\in K }\scp{w_\sigma}{\ResKs}.
\end{gather}

\begin{theorem}
 If $w$ is constant in $K$, $\alpha=0$ is chosen. Otherwise the correction term
  \eqref{eq:correction} with \eqref{eq:error_abgrall} satisfies
  \begin{equation}\label{eq:system}
  \begin{split}
      \sum_{\sigma \in K}r_\sigma^K= 0,
      \qquad
      \sum_{\sigma \in K} \scp{w_\sigma}{r_\sigma^K}= \EE.
  \end{split}
  \end{equation}
  By adding \eqref{eq:correction} to the residual $\ResKs$,
  the resulting scheme using $\tildResKs$ is locally conservative in $u$ and
  entropy conservative.
\end{theorem}
\begin{proof}
  The relation $\eqref{eq:system}$ defines a linear system of equations with
  always at least two unknowns. It is enough to show that \eqref{eq:correction}
  with \eqref{eq:error_abgrall} is a valid solution. 

  The conservation relation for the new scheme is guaranteed because of
  \begin{equation*}
    \sum_{\sigma \in K}r_\sigma^K= \sum_{\sigma \in K}\alpha(w_\sigma -\mean{w} )
    =
    \alpha  \sum_{\sigma \in K}w_\sigma -\sum_{\sigma\in K}  \mean{w} =0.
  \end{equation*}
  The entropy condition is satisfied, since
  \begin{equation}
    \sum_{\sigma \in K} \scp{w_\sigma}{r_\sigma^K} 
    =
    \alpha \sum_{\sigma \in K} \scp{w_\sigma}{(w_\sigma -\mean{w} )}
    =
    \alpha \!
    \sum_{\sigma \in K}\! \scp{w_\sigma-\mean{w}}{w_\sigma -\mean{w}}
    =
    \EE.
  \end{equation}
  It is obvious that $\tildResKs$ fulfills the entropy condition \eqref{eq:entropy_condition}.
  \end{proof}

Again, it should be pointed out that the correction term given in \eqref{eq:correction}
is universally applicable. No further restrictions of the grid structure, point distribution of the DOFs, or the
scheme are needed. It can be applied to any scheme including DG, CG, SUPG, and FR as described in \cite{abgrall2018connection}.
Therefore, it is a universal tool for classical baseline schemes as described above.
Only the error behavior of $\EE$ has to be considered. However, it can be controlled through the applied quadrature, see \cite{abgrall2018general} for details.
To obtain an entropy dissipative scheme, jump or streamline diffusion terms can be
added to the correction term, cf.\ \cite{abgrall2018general}.
In this paper, we will sometimes use the jump diffusion $
\mathcal{J}_\sigma= \theta h_k^2 \oint_{\partial_K} \jump{\nabla \varphi_\sigma} \jump{\nabla w^h}.
 $

\subsection*{Explicit Space-Time Residual Distribution Method}
The DeC method will be used together with the RD framework, resulting in an explicit space-time FE methods \cite{abgrall2017high} with similarities and connections to the modern ADER approach \cite{veiga2020dec}.
The main idea of DeC is based on the Picard Lindelöf theorem;
minimizing the error in a correction algorithm until one reaches the desired
order of accuracy.

To describe the method, the interval between the timesteps  $t^n$ and $t^{n+1}$  will be split into $M$ subintervals given by $t^n = t^{n,0} < t^{n,1} < ... < t^{n,M} = t^{n+1}$.
For each subtimestep $[t^{n,l},t^{n,l+1}]$, there are $K$ corrections
$k = 0,...,M$, where  $U^{n,r,(k)}$ is the $k$-th correction of the $r$-th substep. Furthermore,
\begin{equation}
U^{(k)} := (U^{n,1,(k)},...,U^{n,M,(k)}).
\end{equation}
Then, two operators $\mathcal{L}^1$ and $\mathcal{L}^2$
are introduced. Here,  $\mathcal{L}^1$ is a first order method which is explicit and easy to solve while
 $\mathcal{L}^2$ yields an implicit, high order time integration scheme.
Then, the following DeC procedure is applied.
\begin{enumerate}
\item Set $U^{(0)} = (U^{n,1,0},...,U^{n,M,0}) = (U^n,...,U^n)$.
\item For each correction step $k = 1,...,K$, define $U^{(k+1)}$ as the solution of:
\begin{equation}\label{step2}
\mathcal{L}^1(U^{(k+1)}) = \mathcal{L}^1(U^{(k)})  -  \mathcal{L}^2(U^{(k)}) .
\end{equation}
\item Set $U^{n+1} = U^{n,M,(K)}$.
\end{enumerate}
Finally, the definitions of $\mathcal{L}_\sigma^1$ and $\mathcal{L}_\sigma^2$ are necessary, taking the steady state space residuals $\Phi_{\sigma,x}^K$  into account,
where
 \begin{equation*}
 \begin{gathered}
    \mathcal{L}_\sigma^1(U^{(k)})
    =
 \begin{pmatrix}
|C_\sigma|(U^{n,M,(k)} - U^{n,0}) + \Delta t \beta_M \sum\limits_{K|\sigma \in K} 
 \Phi_{\sigma,x}^K(U^{n,0}) \\\ 
\vdots \\\ 
|C_\sigma|(U^{n,1,(k)} - U^{n,0}) + \Delta t \beta_1\sum\limits_{K|\sigma \in K}
\Phi_{\sigma,x}^K(U^{n,0})
\end{pmatrix},
  \\
      \mathcal{L}_\sigma^2(U^{(k)})
    =
 \begin{pmatrix}
 \sum\limits_{K|\sigma \in K} \left( \oint_K \phi_\sigma \left(U^{n,M,(k)}-U^{n,0} \right)  + \Delta t
 \sum\limits_{r=0}^M \theta_r^M \Phi_{\sigma,x}^K(U^{n,r,(k)}) \right) \\\ 
\vdots \\\ 
\sum\limits_{K|\sigma \in K} \left(\oint_K \phi_\sigma \left(U^{n,1,(k)}-U^{n,0} \right)+ \Delta t
\sum\limits_{r=0}^M \theta_r^1 \Phi_{\sigma,x}^K(U^{n,r,(k)})\right)
\end{pmatrix} .
\end{gathered}
\end{equation*}
Here, $ |C_\sigma|:=\oint_K \phi_\sigma$, and  $ \beta_i$, $\theta_r^{i}$ are the quadrature weights for the time integration.
The $l$-th line of \eqref{step2} simply becomes:
\begin{equation}\label{oneline}
U_\sigma^{n,l,m+1} = U_\sigma^{n,l,m} -
|C_\sigma|^{-1} 
\sum_{K|\sigma \in K}\bigg(\oint_K \varphi_\sigma (U^{n,l,m} - U^{n,0}) + 
\Delta t \sum_{r=0}^M \theta_{r}^{l} \Phi_{\sigma,x}^K(U^{n,r,(k)}) \bigg).
\end{equation}

\begin{remark}
There are two variants of entropy correction terms in \eqref{oneline}.
\begin{enumerate}
\item The correction
term can be applied only to the space residual $ \Phi_{\sigma,x}^K$, which results in a SEC/D-scheme. To obtain the desired FEC/D scheme,  the relaxation approach
\cite{ranocha2020general, abgrall2021relaxation} or  artificial viscosity \cite{glaubitz2016artificial} can be used in this framework. Another possibility would be the application of implicit methods  \cite{kuzmin2020entropy, nordstrom2013summation}.
\item The second possibility is to use the correction term  not only for the space residual, but for the whole bracket, resulting in a FEC scheme.
\end{enumerate}
\end{remark}
However, it should be stressed that this paper is focusing on extending the correction term in the semidiscrete  setting. Therefore, mainly
SEC/D methods will be investigated in numerical simulations especially in combination with RK schemes.
A combination with the relaxation approach \cite{abgrall2021relaxation} can also be done and will be considered in another paper.
Nevertheless, some tests are performed using the second option, demonstrating
the universal applicability of correction terms. Therefore, a possible algorithm
will be described in the following.


The basic idea is to apply the correction not only to the space residual
$\Phi_{\sigma,x}^K(U^{n,r,(k)})$ but to the complete space-time residual
\begin{equation}\label{eq:corr}
\sum_{K|\sigma \in K}\bigg(\oint_K \varphi_\sigma (U^{n,l,m} - U^{n,0}) + 
\Delta t \sum_{r=0}^M \theta_{r}^{l} \Phi_{\sigma,x}^K(U^{n,r,(k)}) \bigg).
\end{equation}
To explain how this works, we shortly describe the algorithm in our update step.
This approach benefits from the fact that we have not applied the method of lines but are running an explicit space-time FE approach. Further, we neglect here the conversion between different
variables (entropy, conservative, control, etc.) and the used coordinated system (reference, Cartesian, barycentric) to simplify the algorithm. Be aware that depending on the applied code, including these can become challenging at least from our personal experience.
The update procedure is
\begin{enumerate}
\item Compute the entropy difference $\eta(U^{(k)})- \eta(U^0)$ at every DOF.
\item Calculate the entropy flux using $U^{(k)}$ at every degree of freedom.
\item Calculate the differences in the entropy in every element $K$ using  the space-time entropy residual $\Phi_{t,x}^{K,e}$.
\item Use the correction term with the calculated entropy differences to correct the space-time residual  \eqref{eq:corr}.
\end{enumerate}
By doing this in every step, the entropy is conserved in space and time and we obtain the desired result.

\subsection{Operator Formulation: Discontinuous Element Based Schemes}
\label{subsec:DG}

Besides the RD formulation, 
another focus of this paper  lies on element based discretizations using
the SBP-SAT framework. Therefore, it will be briefly summarized
and the notations differing from above will be explained.
Again, the domain $\Omega$ is partitioned into non-overlapping subcells 
$\Omega_l \subseteq \Omega$ where the following discrete operators are used
on each element (dropping the elemental index $l$ for convenience).

\begin{itemize}
  \item
  A symmetric and positive definite mass matrix $M$, approximating the $L^2$
  scalar product via $\int_{\Omega_l} u(x) v(x) \dif x = \scp{u}{v}_{L^2(\Omega_l)}
  \approx \scp{u}{v}_M = u^T M v$.

  \item
  Derivative matrices $D_j$, approximating the partial derivative $\partial_j u
  \approx D_j u$.

  \item
  A restriction/interpolation operator $R$, performing interpolation to the
  boundary nodes at $\partial\Omega_l$ via $R u$.

  \item
  A symmetric and positive definite boundary mass matrix $B$, approximating
  the scalar product on $L^2(\partial\Omega_l)$.

  \item
  Multiplication operators $N_j$, $j \in \set{1, \dots, d}$, representing
  the multiplication of functions on the boundary $\partial\Omega_l$ by the
  $j$-th component $\nu_j$ of the outer unit normal $\nu$ at $\partial\Omega_l$.
\end{itemize}
Together, the restriction and boundary operators approximate the boundary
integral with respect to the outer unit normal as in the divergence theorem, i.e.
\begin{equation*}
  u^T R^T B N_j R v
  \approx
  \int_{\partial\Omega} u \, v \, \nu_j.
\end{equation*}
If the SBP property
\begin{equation}
\label{eq:SBP}
  M D_j + D_j^T M
  =
  R^T B N_j R
\end{equation}
is fulfilled, the divergence theorem is mimicked on a discrete
level \cite{fernandez2014generalized}.
Using the SBP property \eqref{eq:SBP}, one can transfer
stability results established at the continuous level to the discrete level,
cf.\ \cite{svard2014review,chen2020review} and references cited therein.
The general semidiscretizations considered here can be written as
\begin{equation}
\label{eq:operator-semidiscretisation}
  \partial_t u
  =
  \VOL + \SURF,
\end{equation}
where $\VOL$ are volume terms discretizing $\partial_j f^j(u)$ and $\SURF$ are
surface terms implementing interface/boundary conditions weakly. For the $i$-th
conserved variable $u_i$, $i \in \set{1, \dots, m}$, the corresponding
semidiscretization is $\partial_t u_i = \VOL_i + \SURF_i.$

\begin{example}
  A central nodal DG scheme using the numerical (surface) fluxes $\fnumj$ can be
  obtained by choosing
  \begin{equation*}
    \VOL = -D_j f^j,
    \quad
    \SURF = - M^{-1} R^T B N_j (\fnumj - R f^j).
  \end{equation*}
  This is nothing else than a classical nodal DG formulation as described in  \cite{hesthaven2007nodal}.
\end{example}

\begin{example}
  A flux differencing or split form discretization using symmetric two-point numerical
  volume fluxes $\fvolj$ and numerical surface fluxes $\fnumj$ can be obtained
  following \cite{fisher2013high,gassner2016split} as
  \begin{equation}\label{eq:flux_splitting}
    \VOL^{(m)} = - 2 \sum_{j=1}^d \sum_k {D_j}_{m,k} \fvolj\bigl( u^{(m)}, u^{(k)} \bigr),
    \;
    \SURF = - M^{-1} R^T B N_j (\fnumj - R f^j),
  \end{equation}
  where the upper indices $(m), (k)$ indicate the grid node.
  The properties of the schemes depend strongly on the
  numerical fluxes
  and a lot of effort has been made to construct
  numerical fluxes with desirable properties \cite{ranocha2018comparison}.
  Further, a close connection to the nodal DG framework (or FD) can be seen using \eqref{eq:flux_splitting}.
  Actually, this approach will lead to a split (or skew-symmetric)
  semidiscretization of a nodal DG scheme. Examples
  and further explications can be found in
  \cite{ranocha2017shallow, ranocha2018thesis}
  and later in this article.
\end{example}
Up to now, no additional conditions on the semidiscretization \eqref{eq:operator-semidiscretisation}
have been formulated.  Similar to the RD 
setting described, the focus lies on \emph{local conservation} and \emph{entropy conservation}.
\begin{itemize}
 \item 
The discretization should be locally conservative:
\begin{equation}
\label{eq:operator-conservation-target}
  \forall i \in \set{1, \dots, m}\colon
  \qquad
  1^T M \partial_t u_i
  =
  - \sum_{j}1^T R^T B N_j \fnumj_i,
\end{equation}
where the indices stands for 
the $i$-th 
  conserved variable  in the $j$-th
direction.
\item An entropy conservative semidiscretization fulfils 
\begin{equation}
\label{eq:operator-entropy-target}
  w_i^T M \partial_t u_i
  =
  - \sum_j1^T R^T B N_j \Fnumj,
\end{equation}
where 
 $\Fnumj$ are
numerical entropy fluxes corresponding to $\fnumj$
\cite{tadmor1987numerical,tadmor2003entropy} and the  mass matrix $M$ is diagonal, which will assumed in the following.
\end{itemize}

The basic idea of \cite{abgrall2018general} will be embedded in
this framework. The key is to enforce
\eqref{eq:operator-entropy-target} for any semidiscretization via the addition of a
correction term $\r$ that is consistent with zero and does not violate the conservation
relation \eqref{eq:operator-conservation-target}, resulting in
\begin{equation}
\label{eq:operator-semidiscretisation-corrected}
  \partial_t u_i
  =
  \VOL_i + \SURF_i + \r_i.
\end{equation}
Using the mass matrix $M$ for discrete integration
\cite{ranocha2020strong}, the correction term is
\begin{equation}
\label{eq:operator-entropy-correction}
\begin{gathered}
  \r_i = \alpha \left(
    w_i - \frac{1^T M w_i}{1^T M 1} 1
  \right),
  \quad
  \alpha = \frac{\mathcal{E}}{w_k^T M w_k
                        - \frac{(1^T M w_k) (1^T M w_k)}{1^T M 1}},
  \\
  \mathcal{E} = - 1^T R^T B N_j \Fnumj - w_k^T M \, \VOL_k - w_k^T M \, \SURF_k.
\end{gathered}
\end{equation}
If the denominator of $\alpha$ in \eqref{eq:operator-entropy-correction} is zero,
the numerical solution is constant in the element because of the Cauchy Schwarz
inequality (in that case, $1$ and $w_k$ are linearly dependent for each $k$).
If the scheme is chosen to have a continuous behavior over
the element boundaries as in continuous Galerkin or SUPG methods,
no further considerations are necessary since the constraints
\eqref{eq:operator-conservation-target} and
\eqref{eq:operator-entropy-target} do not contradict each other.
No special care has to be taken and the correction is always possible.
However, one has to be more careful for DG schemes, as described in the
following remark.

\begin{remark}\label{re:Tadmor_restriction}
\label{rem:EC-fluxes-necessary}
  Possible contradictions of \eqref{eq:operator-conservation-target}
  and \eqref{eq:operator-entropy-target} can be studied in the setting of
  EC schemes in the sense of Tadmor.
  There can only be
  problems if the denominator of $\alpha$ vanishes. In that case, all
  $w_i$ are proportional to $1$ inside an element and the scheme
  reduces to a finite volume scheme using the numerical fluxes $\fnumj$
  for such an element:  \eqref{eq:operator-entropy-target} has to hold.
  This is satisfied for the schemes investigated in this section, if the
  numerical surface fluxes $\fnumj$ are EC and $\Fnumj$
  the corresponding numerical entropy fluxes in the sense of Tadmor,
  i.e. if
  \begin{gather}
    (w_{i,+} - w_{i,-}) \fnumj_i(w_-, w_+) = \psi^j_+ - \psi^j_-,
    \\
    \Fnumj = \frac{w_{i,+} + w_{i,-}}{2} \fnumj_i(w_-, w_+) - \frac{\psi^j_+ + \psi^j_-}{2},
  \end{gather}
  where $\psi^j$ is the flux potential $\psi^j = w_i f^j_i - F^j$.
  Indeed, if the numerical solution is constant in an element, DG type schemes
  such as in Example~\ref{rem:EC-fluxes-necessary} result in a finite volume
  scheme
  \begin{equation*}
    \partial_t u = - M^{-1} R^T B N_j (\fnumj - R f^j),
  \end{equation*}
  which is conservative because of $1^T R^T B N_j R f^j = 0$ since the divergence
  theorem has to hold for constants because of consistency. Additionally,
  \begin{multline*}
    w_i^T M \partial_t u_i
    =
    - w_i^T R^T B N_j \fnumj_i + w_i^T R^T B N_j R f^j_i
    \\
    =
    - w_i^T R^T B N_j \fnumj_i + 1^T R^T B N_j R \psi^j.
  \end{multline*}
  Hence, it suffices to consider one boundary node. There,
  \begin{equation}\label{eq:Tadmor}
  \begin{aligned}
    \psi^j_\pm - w_{i,\pm} \fnumj_i(w_-, w_+)
    &=
    \psi^j_\pm
    - \left(
      \frac{w_{i,+} + w_{i,-}}{2} \pm \frac{w_{i,+} - w_{i,-}}{2}
    \right) \fnumj_i(w_-, w_+)
    \\
    &=
    \psi^j_\pm - \frac{w_{i,+} + w_{i,-}}{2} \fnumj_i(w_-, w_+)
    \mp \frac{\psi^j_+ - \psi^j_-}{2}
    \\
    &=
    \frac{\psi^j_+ + \psi^j_-}{2} - \frac{w_{i,+} + w_{i,-}}{2} \fnumj_i(w_-, w_+)
    =
    - \Fnumj.
  \end{aligned}
  \end{equation}
\end{remark}

\begin{remark}
 Not only Tadmor's framework yields a solution to this.
 One can also take  $\Fnumj= \{w \}\cdot \fnumj- \psi^j (\{ w\})$.
 At it is shown in \cite{abgrall2018general} by a simple Taylor expansion analysis, there is no contradiction using instead this interpretation.
\end{remark}


\begin{remark}\label{re:no_problem}
Finally, it should be mentioned that 
  the constraints \eqref{eq:operator-conservation-target} and
  \eqref{eq:operator-entropy-target} can only contradict each other if
  the numerical solution is constant inside an element. In that case,
  one can also decide to drop the EC constraint
  \eqref{eq:operator-conservation-target} since a reasonable baseline
  scheme should give acceptable results. Then, no special choice of numerical
  fluxes at the boundaries is necessary.
\end{remark}

\begin{theorem}
\label{thm:optimization-operator-entropy}
  Let $\mathcal{E} \neq  0$ and  $w$ be not constant.
  The correction term $\r = (\r_1, \dots, \r_m)$ \eqref{eq:operator-entropy-correction} is the unique
solution of the 
  the minimization problem 
    \begin{equation}
  \label{eq:optimization-operator-entropy}
    \min_{\r} \frac{1}{2} \norm{\r}_M^2
    \quad
    \st 1^T M \r_i = 0,\; w_i^T M \r_i = \mathcal{E},
  \end{equation}
  with $\norm{\r}_M^2 = \r_i^T M \r_i$
  such that \eqref{eq:operator-conservation-target}
  and \eqref{eq:operator-entropy-target} are satisfied.
\end{theorem}

\begin{proof}
  Equation \eqref{eq:optimization-operator-entropy} can be reformulated as
  \begin{equation*}
    \min_{\r} \frac{1}{2} \r^T (\I_m \otimes M) \r
    \;
    \st A \r = b,
    \quad
    A = \begin{pmatrix} \I_m \otimes (1^T M) \\ w^T (\I_m \otimes M) \end{pmatrix},
    b = \begin{pmatrix} 0 \\ \mathcal{E} \end{pmatrix}.
  \end{equation*}
  Since $M$ is symmetric \& positive definite, the unique solution satisfies
  \cite[Section~16.1]{nocedal1999numerical}
  \begin{equation*}
    \begin{pmatrix}
      (\I_m \otimes M) & -A^T \\
      A & 0
    \end{pmatrix}
    \begin{pmatrix}
      \r \\
      \lambda
    \end{pmatrix}
    =
    \begin{pmatrix}
      0 \\
      b
    \end{pmatrix}
  \end{equation*}
  for some $\lambda \in \R^{m+1}$. Hence, $\r$ must satisfy the $m+1$ constraints
  $1^T M \r_i = 0, w^T M \r = \mathcal{E}$ and $M \r_i$ must be in the span of
  $\set{M 1, M w_i}$ such that the coefficients of $M w_i$ are the same for all
  $i \in \set{1, \dots, m}$. This is obviously true for $\r_i$ as defined in
  \eqref{eq:operator-entropy-correction}.
\end{proof}
\begin{remark}
In the following, we write
``The constraints do not contradict each other''
to make clear that the minimization problem has a unique solution.
To specify this more precise, we assume that $w$ is not constant.
\end{remark}

Using the same argument used in the proof of
Theorem~\ref{thm:optimization-operator-entropy}, one obtains
\begin{proposition}
If the constraints do not contradict each other,
  the correction term $r_\sigma^K$ \eqref{eq:correction} is the unique optimal correction
  of \eqref{eq:residual_corr}, measured in the discrete norm induced by the identity
  matrix $\I$, such that \eqref{eq:entropy_condition} and
  $\sum_{\sigma \in K} r_\sigma^K = 0$ are satisfied, i.e. $r = (r_\sigma^K)_\sigma$
  \eqref{eq:correction} is the solution of
  \begin{equation}
  \label{eq:optimization-RD-entropy}
    \min_{r} \frac{1}{2} \norm{r}_{\I}^2
    \quad
    \st 1^T \I r_\sigma^K = 0,\; w_\sigma^T \I r_\sigma^K = \mathcal{E},
    \qquad
    \norm{r}_{\I}^2 = (r_\sigma^K)^T \I r_\sigma^K.
  \end{equation}
\end{proposition}

  There are some differences between the role of the correction terms $\r_i$
  described in this section and the terms $r_\sigma^K$ used in the previous section.
  Firstly, since the correction $\r_i$ \eqref{eq:operator-entropy-correction} is
  added to the other side of the hyperbolic equation, the sign differs from the
  correction term $r_\sigma^K$ \eqref{eq:correction} of \cite{abgrall2018general}.
  Secondly, $\r_i$ \eqref{eq:operator-entropy-correction} is a correction for the
  pointwise time derivative of $u$ while $r_\sigma^K$ \eqref{eq:correction} is a
  correction for an integrated version. Loosely speaking, they are related via
  \begin{equation}\label{eq:approx_two}
    r_\sigma^K  \simeq M \r_i.
  \end{equation}
  Additionally, the role of the indices differs: $\r_i$ is a correction term
  for the $i$-th variable at all grid nodes while $r_\sigma^K$ is a correction term
  at the grid node $\sigma$ for all variables (if a DG setting is used for
  the RD scheme).
  
  Using the notation of this section, the correction term $r_\sigma^K$ \eqref{eq:correction} corresponds to
  \begin{equation}
  \label{eq:correction-abgrall-as-operator}
  \begin{gathered}
    \r_i = \alpha \left(
      M^{-1} w_i - \frac{1^T w_i}{1^T 1} M^{-1} 1
    \right),
    \quad
    \alpha = \frac{\mathcal{E}}{w_k^T w_k
                          - \frac{(1^T w_k) (1^T w_k)}{1^T 1}},
    \\
    \mathcal{E} = 1^T R^T B N_j \Fnumj - w_k^T \Phi_k.
  \end{gathered}
  \end{equation}
  While $r_\sigma^K$ \eqref{eq:correction} is the optimal correction with respect
  to the identity matrix, its corresponding pointwise correction
  $\r_i$ \eqref{eq:correction-abgrall-as-operator} is optimal with respect to
  the norm induced by $M$.

  Using the notation of RD schemes, the entropy correction term
  $\r_i$ \eqref{eq:operator-entropy-correction} corresponds to
  \begin{equation}
  \label{eq:operator-entropy-correction-RD}
  \begin{gathered}
    r_\sigma^K = \alpha \left(
      \oint_K w_\sigma \phi_\sigma
      - \frac{\sum_{\rho \in K} \oint_K w_\rho \phi_\rho}
             {\sum_{\rho \in K} \oint_K \phi_\rho}
        \oint_K \phi_\sigma
    \right)
    =
    \alpha \left( w_{\sigma}
      - \frac{\sum_{\rho \in K} \oint_K w_\rho \phi_\rho}
             {\sum_{\rho \in K} \oint_K \phi_\rho}
    \right) \oint_K \phi_\sigma,
    \\
    \alpha = \frac{
      - \oint_{\partial K} \Fnumj\bigl(w_{|K}^h, w_{|K^-}^h \bigr) \nu_j
      + \sum_{\sigma\in K }\scp{w_\sigma}{\ResKs}
    }{
      \sum_{\sigma,\rho \in K} \oint_K \abs{w_\sigma}^2 \phi_\sigma
      - 
      |K|^{-1} \left( \sum_{\sigma \in K} \oint_K (w_\sigma \phi_\sigma) \right) \cdot \left( \sum_{\rho \in K} \oint_K (w_\rho \phi_\rho) \right)
    },
  \end{gathered}
  \end{equation}
  in accordance with \eqref{eq:approx_two}.
  Hence, \eqref{eq:operator-entropy-correction} uses an integral weighting (by the
  quadrature rule) instead of a summation without weights.
  While $\r_i$ \eqref{eq:operator-entropy-correction} is the optimal correction
  with respect to the mass matrix $M$, its corresponding integral correction
  $r_\sigma^K$ \eqref{eq:operator-entropy-correction-RD} is optimal with respect to
  the norm induced by $M^{-1}$.
  
%
  We would like to stress that \eqref{eq:operator-entropy-correction} and
  \eqref{eq:correction} are explicit solutions of the corresponding
  optimization problems \eqref{eq:optimization-operator-entropy} and
  \eqref{eq:optimization-RD-entropy}, respectively. No optimization solver
  is necessary to solve these problems. Hence, the computational cost of
  the new approach using a weighting by the quadrature rule is basically the
  same as for the approach suggested in \cite{abgrall2018general}.

\begin{remark}\label{rem:singh2019kinetic}
By Theorem~\ref{thm:optimization-operator-entropy}, the entropy correction
terms can be interpreted as a solution of a quadratic minimization problem with
equality constraints. The idea of solving such a problem can also be exploited
for many different applications.
In \autoref{sec:grid-refinement}, an application of grid refinement and coarsening
is presented.
A combination of split forms and correction terms similar to the ones described here
has been presented for the kinetic energy for the Euler equations in \cite{singh2019kinetic} .
\end{remark}

\subsection{Finite Difference and Global Spectral Collocation Schemes}
\label{subsec:FD-SC}

Classical single block finite difference and spectral collocation schemes can be
interpreted as RD or DG schemes described in Sections~\ref{subsec:RD} and \ref{subsec:DG}
with one element. In that case, the entropy corrections described above yield
globally conservative and globally EC/D schemes.

Sadly, \emph{global conservation} and a \emph{global entropy inequality} do not imply any sort
of convergence towards an entropy solution of scalar conservation laws, even if
the scheme converges. This will be demonstrated by the following example.
\begin{example}
  Consider Burgers' equation $\partial_t u + \partial_x u^2/2 = 0$ on $[0,3]$
  with periodic boundary conditions and the initial condition 
  \begin{equation*}
    u_0(x)
    =
    \begin{cases}
      -1, & \text{if } 1 < x < 2, \\
      +1, & \text{else}.
    \end{cases}
  \end{equation*}
  The unique entropy solution contains a stationary shock at $x = 1$ and a
  rarefaction wave starting at $x = 2$.

  Central periodic finite difference and Fourier collocation schemes can be
  represented by a skew-symmetric derivative operator $D$ and mass matrix
  $M \propto \I$. Hence, there is no difference between the approaches
  \eqref{eq:correction} and \eqref{eq:operator-entropy-correction}.
  Since $u^2$ is constant, a classical central scheme yields a stationary numerical
  solution. The entropy correction vanishes, too, since $\mathcal{E}$ is zero
  (because of periodic boundary conditions). Hence, the same stationary numerical
  solution is obtained if the entropy correction is added.
  If an element based scheme is used instead, the numerical solution with entropy
  correction term cannot be stationary if the grid is fine enough (obtained by
  increasing the number of elements), since the difference of numerical entropy
  boundary fluxes does not vanish if the element contains exactly one initial
  discontinuity.
\end{example}

\section{Generalizations}
\label{sec:generalizations}
In this section, some generalizations of the entropy correction terms based on
the interpretations as a quadratic optimization problem will be developed. Here,
the notation of Section~\ref{subsec:DG} for discontinuous element based schemes
will be used. Please, keep in mind that in case of DG schemes,
the issues described in Remarks \ref{re:Tadmor_restriction}--\ref{re:no_problem}
have to be taken into account.

\subsection{Inequality Constraints}

In many applications, the main interest lies in
an entropy \emph{inequality} instead of EC schemes, resulting in
some kind of stability estimates. For example, even if the baseline scheme is
not necessarily EC/D in general, it can be dissipative in some
cases. Then, it could be beneficial to preserve this dissipation introduced
by the baseline scheme. Moreover, it could be possible to obtain better approximations
with smaller corrections if some entropy dissipation is allowed. Instead of
\eqref{eq:optimization-operator-entropy} in Theorem~\ref{thm:optimization-operator-entropy},
such an optimization problem is
\begin{equation}
\label{eq:optimization-operator-entropy-inequality}
  \min_{\r} \frac{1}{2} \norm{\r}_M^2
  \quad
  \st 1^T M \r_i = 0,\; w_i^T M \r_i \leq \mathcal{E}.
\end{equation}
While a solution of \eqref{eq:optimization-operator-entropy-inequality} is still
conservative, i.e. \eqref{eq:operator-conservation-target} holds, the entropy
inequality
\begin{equation*}
  w_i^T M \partial_t u_i
  \leq
  - 1^T R^T B N_j \Fnumj
\end{equation*}
holds instead of the entropy equality \eqref{eq:operator-entropy-target}.
The next theorem simply states now  that
  the semidiscretization \eqref{eq:operator-semidiscretisation-corrected} obtained
  by the correction solving \eqref{eq:optimization-operator-entropy-inequality}
  is given by the unmodified method if it is entropy dissipative and by the
  entropy conservative scheme \eqref{eq:operator-entropy-correction} if the
  baseline scheme produces spurious entropy (per element).
 Its proof is just a paraphrase of this observation.
\begin{theorem}
\label{thm:optimization-operator-entropy-inequality}
  If the constraints do not contradict each other, the optimization problem
  \eqref{eq:optimization-operator-entropy-inequality}
  has a unique solution $\r$, which is given by $\r = 0$ if $\mathcal{E} > 0$
  and \eqref{eq:operator-entropy-correction} if $\mathcal{E} \leq 0$.
\end{theorem}
%

\subsection{Multiple Constraints}

A generalization of the approach of
Theorem~\ref{thm:optimization-operator-entropy} to multiple linear constraints
is straightforward. This is demonstrated for two constraints
\begin{equation}
\label{eq:operator-2-entropies-target}
  w_i^T M \r_i = \mathcal{E},
  \quad
  \widetilde{w}_i^T M \r_i = \widetilde{\mathcal{E}}
\end{equation}
in addition to \eqref{eq:operator-conservation-target}. Here, $\widetilde{\mathcal{E}}$
is the difference of desired and current values for a linear constraint similar
to $\mathcal{E}$ in \eqref{eq:operator-entropy-correction}.
For example, $\widetilde{\mathcal{E}}$ could be caused by a correction for the
kinetic energy for the Euler equations, cf.\ $\mathcal{E}$ in
\eqref{eq:operator-Ekin-correction} below.

\begin{theorem}
\label{thm:optimization-operator-2-entropies}
  If the constraints \eqref{eq:operator-conservation-target} and
  \eqref{eq:operator-2-entropies-target} do not contradict each other,
  the unique solution $\r = (\r_1, \dots, \r_m)$ of
  \begin{equation}
  \label{eq:optimization-operator-2-entropies}
    \min_{\r} \frac{1}{2} \norm{\r}_M^2
    \quad
    \st 1^T M \r_i = 0,\; w_i^T M \r_i = \mathcal{E}, \;
    \widetilde{w}_i^T M \r_i = \widetilde{\mathcal{E}},
  \end{equation}
  is given by
  $  \r_i
    =
    \alpha \left( w_i - \frac{1^T M w_i}{1^T M 1} 1 \right)
    + \widetilde{\alpha} \left( \widetilde{w}_i - \frac{1^T M \widetilde{w}_i}{1^T M 1} 1 \right),$
  where
  \begin{equation*}
    \begin{pmatrix}
      \alpha \\
      \widetilde{\alpha}
    \end{pmatrix}
    =
    \begin{pmatrix}
      w_i^T M w_i - \frac{(1^T M w_i) (1^T M w_i)}{1^T M 1} &
      w_i^T M \widetilde{w}_i - \frac{(1^T M w_i) (1^T M \widetilde{w}_i)}{1^T M 1} \\
      \widetilde{w}_i^T M w_i - \frac{(1^T M \widetilde{w}_i) (1^T M w_i)}{1^T M 1} &
      \widetilde{w}_i^T M \widetilde{w}_i - \frac{(1^T M \widetilde{w}_i) (1^T M \widetilde{w}_i)}{1^T M 1}
    \end{pmatrix}^{-1}
    \begin{pmatrix}
      \mathcal{E} \\
      \widetilde{\mathcal{E}}
    \end{pmatrix}.
  \end{equation*}
\end{theorem}
\begin{proof}
  Equation \eqref{eq:optimization-operator-2-entropies} can be reformulated as
  \begin{equation*}
    \min_{\r} \frac{1}{2} \r^T (\I_m \otimes M) \r
    \;
    \st A \r = b,
    \quad
    A = \begin{pmatrix} \I_m \otimes (1^T M) \\ w^T (\I_m \otimes M) \\ \widetilde{w}^T (\I_m \otimes M) \end{pmatrix},
    b = \begin{pmatrix} 0 \\ \mathcal{E} \\ \widetilde{\mathcal{E}} \end{pmatrix}.
  \end{equation*}
  As in the proof of Theorem~\ref{thm:optimization-operator-entropy}, $\r$ must
  satisfy the constraints $1^T M \r_i = 0, w^T M \r = \mathcal{E},
  \widetilde{w}_i^T M \r_i = \widetilde{\mathcal{E}}$ and $M \r_i$ must be in the span of
  $\set{M 1, M w_i, M \widetilde{w}_i}$ such that the coefficients of $M w_i$
  and $M \widetilde{w}_i$ are the same for all $i \in \set{1, \dots, m}$, respectively.
  Hence,
   $ \r_i = \alpha w_i + \widetilde{\alpha} \widetilde{w}_i + c_i$,
  where
 $ 
    1^T M \r_i = 0$, i.e.
 $   \alpha 1^T M w_i
    + \widetilde{\alpha} 1^T M \widetilde{w}_i
    + c_i 1^T M 1 = 0$
   leads to $$
    c_i = - \frac{\alpha 1^T M w_i + \widetilde{\alpha} 1^T M \widetilde{w}_i}{1^T M 1}.$$
  Finally, $\alpha$ and $\widetilde{\alpha}$ have to solve
  \begin{equation*}
  \begin{aligned}
    \left( w_i^T M w_i - \frac{(1^T M w_i) (1^T M w_i)}{1^T M 1} \right) \alpha
    + \left( w_i^T M \widetilde{w}_i - \frac{(1^T M w_i) (1^T M \widetilde{w}_i)}{1^T M 1} \right) \widetilde{\alpha}
    &= \mathcal{E},
    \\
    \left( \widetilde{w}_i^T M w_i - \frac{(1^T M \widetilde{w}_i) (1^T M w_i)}{1^T M 1} \right) \alpha
    + \left( \widetilde{w}_i^T M \widetilde{w}_i - \frac{(1^T M \widetilde{w}_i) (1^T M \widetilde{w}_i)}{1^T M 1} \right) \widetilde{\alpha}
    &= \widetilde{\mathcal{E}},
  \end{aligned}
  \end{equation*}
  proving the assertion.
\end{proof}

\begin{remark}
  $ $
  \begin{enumerate}
  \item It should be stressed again that no optimization
  solver is necessary.
\item   A similar statement, omitted, can be given for the RD formulations.
\item 
  It can be desirable to satisfy entropy (in-) equalities for multiple entropies.
  Based on Remark~\ref{re:Tadmor_restriction}, the numerical surface fluxes
  $\fnumj$ should be EC for both entropies. However, this is in
  general not possible in case of DG schemes and following Tadmor's framework.
  Indeed, for a scalar conservation law and a fixed entropy,
  the EC numerical flux is uniquely determined as
  $\fnumj = {\jump{\psi^j}}/{\jump{w}}$.
 In the case of CG, it seems possible since no constraints are formulated.
 However, a decrease of accuracy may be expected also through the results 
 of Osher for E-schemes \cite{osher1984riemann}, which satisfy an entropy
 inequality for every convex entropy but are at most first order accurate.
 These investigations are left for future research.
 \end{enumerate}
\end{remark}

\subsection{Kinetic Energy for the Euler Equations}
\label{subsec:kinetic-energy}

Consider the compressible Euler equations in two space dimensions (the extension
to three space dimensions is straightforward)
\begin{equation}
\label{eq:Euler}
\begin{aligned}
  \partial_t
  \underbrace{
  \begin{pmatrix}
    \rho
    \\
    \rho v_x
    \\
    \rho v_y
    \\
    \rho e
  \end{pmatrix}
  }_{= u}
  + \,\partial_x
  \underbrace{
  \begin{pmatrix}
    \rho v_x
    \\
    \rho v_x^2 + p
    \\
    \rho v_x v_y
    \\
    (\rho e + p) v_x
  \end{pmatrix}
  }_{= f^x(u)}
  + \,\partial_y
  \underbrace{
  \begin{pmatrix}
    \rho v_y
    \\
    \rho v_x v_y
    \\
    \rho v_y^2 + p
    \\
    (\rho e + p) v_y
  \end{pmatrix}
  }_{= f^y(u)}
  =
  0,
\end{aligned}
\end{equation}
where $\rho$ is the density of the gas, $v = (v_x,v_y)$ its speed, $\rho v$ the
momentum, $e$ the specific total energy, and $p$ the pressure. The total energy
$\rho e$ can be decomposed into the internal energy $\rho \epsilon$ and the kinetic
energy $\Ekin = \frac{1}{2} \rho v^2$, i.e. $\rho e = \rho \epsilon + \frac{1}{2}
\rho v^2$. For a perfect gas,
$  p
  = (\gamma-1) \rho \epsilon
  = (\gamma-1) \big( \rho e - \frac{1}{2} \rho v^2 \big),
$ 
where $\gamma$ is the ratio of specific heats. For air, $\gamma = 1.4$ will be
used, unless stated otherwise.
The kinetic energy satisfies  
\begin{equation}
\label{eq:Ekin}
  - \partial_t \Ekin
  =
  - \frac{1}{2} v^2 \partial_t \rho
  + v \cdot \partial_t (\rho v)
  =
  \partial_j \left( \frac{1}{2} \rho v^2 v_j + p v_j \right)
  - p \partial_j v_j.
\end{equation}
A numerical flux $\fnumj$ is \emph{kinetic energy preserving} (KEP)
\cite{Jameson,ranocha2020entropy,ranocha2020preventing}, if
\begin{equation}
\label{eq:fnumj-KEP}
  \fnumj_{\rho v_i} = \armean{v_i} \fnumj_{\rho} + \armean{p} \delta_{ij}.
\end{equation}
The corresponding numerical "flux" for the kinetic energy (approximating the conservative
part of the kinetic energy equation) is
\begin{equation*}
  \Fnumj(u^-, u^+)
  =
  \frac{1}{2} v_i^- v_i^+ \fnumj_\rho(u^-, u^+)
  + \frac{p^+ v_j^- + p^- v_j^+}{2}.
\end{equation*}
Using $w = \partial_u \Ekin(u)$, a KEP semidiscretization
mimicking \eqref{eq:Ekin} has to satisfy (cf.\ \cite[Section~7.4]{ranocha2018thesis})
\begin{equation}
\label{eq:operator-Ekin-target}
  1^T M \partial_t \Ekin
  =
  p^T M D_j v_j
  - 1^T R^T B N_j \bigl( \Fnumj - \Snumj \bigr),
\end{equation}
where the discretisation $\Snumj$ of the nonconservative term $- p \partial_j v_j$
at the surface between two elements is given as
\begin{equation*}
  \Snumj(u^-, u^+)
  =
  p^- \frac{v_j^+ - v_j^-}{2}.
\end{equation*}
Here, the argument $u^-$ comes from the interior of an element and the argument
$u^+$ from the neighboring element.

As mentioned in Remark~\ref{rem:singh2019kinetic}, correction terms have been
used to obtain KEP schemes in \cite{singh2019kinetic}.
In contrast to the approach presented in the following, a certain split form of
the Euler equations has been used there instead of a central discretization and Abgrall's
correction terms are used to remove some interpolation errors at the boundaries.

Using the same approach as in Section~\ref{subsec:DG} results in semidiscretizations
\eqref{eq:operator-semidiscretisation-corrected}, where the correction term $\r_i$
has to be chosen such that local conservation \eqref{eq:operator-conservation-target}
and kinetic energy preservation \eqref{eq:operator-Ekin-target} are satisfied.

\begin{remark}
  Similarly to Remark~\ref{re:Tadmor_restriction}, the constraints
  \eqref{eq:operator-conservation-target} and \eqref{eq:operator-Ekin-target}
  do not contradict each other in a finite volume setting if KEP
  fluxes $\fnumj$ are used.
\end{remark}

The correction term for the kinetic energy is
\begin{equation}
\label{eq:operator-Ekin-correction}
\begin{gathered}
  \r_i = \alpha \left(
    w_i - \frac{1^T M w_i}{1^T M 1} 1
  \right),
  \quad
  \alpha = \frac{\mathcal{E}}{w_k^T M w_k
                        - \frac{(1^T M w_k) (1^T M w_k)}{1^T M 1}},
  \\
  \mathcal{E} = p^T M D_j v_j - 1^T R^T B N_j \bigl( \Fnumj - \Snumj \bigr)
  - w_k^T M \, \VOL_k - w_k^T M \, \SURF_k.
\end{gathered}
\end{equation}
\begin{proposition}
  If the constraints \eqref{eq:operator-conservation-target} and
  \eqref{eq:operator-Ekin-target} do not contradict each other,
  the correction term $\r$ \eqref{eq:operator-Ekin-correction} is the unique
  optimal correction of \eqref{eq:operator-semidiscretisation-corrected}, measured
  in the discrete norm induced by $M$, such that \eqref{eq:operator-conservation-target}
  and \eqref{eq:operator-Ekin-target} are satisfied.
\end{proposition}

\begin{remark}
$ $
\begin{enumerate}
\item 
Again,  \eqref{eq:operator-Ekin-correction}
  is an explicit solution, no optimization solver is necessary.
\item
  Using Theorem~\ref{thm:optimization-operator-2-entropies}, combined correction
  terms for the entropy and kinetic energy can be created for the Euler equations.
  The corresponding entropy is chosen as
 $   U = - \tfrac{\rho s}{\gamma - 1}$, with 
$    s = \log\bigl(p / \rho^\gamma \bigr)$, and potentials and entropy variables are defined accordingly.
  \end{enumerate}
\end{remark}

\section{Numerical Examples}
\label{sec:numerical-examples}

In this section, some numerical examples using the correction terms
are presented for several types of schemes.
We concentrate on different kinds of schemes:
the continuous Galerkin schemes \cite{abgrall2017high, abgrall2019analysis},
the $\emph{psi}$-scheme \cite{abgrall2019high}, and the DG method
\cite{chen2017entropy}. The first two methods use the RD framework and the third one
is described and implemented  in the operator formulation from subsection \ref{subsec:DG}.
We use RK-schemes or the DeC approach to make the numerical methods fully discrete.
We would like to emphasize that this paper focuses on the analysis of semidiscrete EC/D schemes and the entropy correction terms. Finally,  some simulations are presented using FEC/D-methods and compared to the SEC/D setting.
Before, starting it should be marked that EC/D is not the holy grail in numerical analysis,
meaning that a bad baseline scheme will always behave inadequately even if one forces the scheme to be EC/D. An example for this can be found in the Appendix~\ref{eq:advection}. Therefore, one should already start with a scheme which  at least performs for simple problems in an adequate way.

\subsection{Two-Dimensional Scalar Equations}

Here, the focus lies on a comparison between the correction terms
\eqref{eq:correction} and \eqref{eq:operator-entropy-correction-RD}
using different weightings. In \eqref{eq:correction}, the identity matrix is used
whereas \eqref{eq:operator-entropy-correction-RD} applies the mass matrix $M$.
A first comparison is given using these two approaches. It is clear that the difference should be  quite small. However, to have a closer look on the
behavior, two examples are demonstrated.
Here, for comparison  a  pure continuous Galerkin scheme in the RD
setting \cite{abgrall2019analysis} is considered.

\subsubsection*{Rotation}

The first problem is a linear rotation equation in two space dimensions
given by
\begin{equation}\label{eq:linear_advec_two}
\begin{aligned}
      \partial_t u(t,x,y) + \partial_x(2\pi y u(t,x,y)) + \partial_y(2\pi x u(t,x,y)) =0&,
      && (x,y) \in D, t\in (0, 1), \\
     u(0,x,y)=u_0(x,y)=\exp\bigl( -40(x^2+(y-0.5)^2) \bigr)&,&& (x,y)\in D,
\end{aligned}
\end{equation}
where $D$ is the unit disk in $\R^2$.
For the boundary, outflow conditions are considered. For time integration,
the fourth order strong stability preserving scheme SSPRK(5,4) is used  in the RD framework with 
CFL number $0.4$.  The correction  is  done at each step in the semidiscrete setting. As in \cite{abgrall2019analysis, abgrall2019analysis_2}  a slight decrease of the order can be
recognized that is not due to the usage of the entropy correction terms and known in continuous FE discretization by coupling space and time.
A pure continuous Galerkin scheme of fourth order is used and Bernstein polynomials
are applied as basis functions, see \cite{abgrall2019high}.
In this test, a small bump located around $(0,0.5)$
is moving around in a circle. The rotation is completed at $t = 1$.
The mesh contains 3582 triangular elements.  In Table~\ref{entrokey},
the change in the energy is given after half rotation and after one full rotation.

\begin{table}[!ht]
\centering
  \caption{Total energy  change  $\int_\Omega U^2_{corr}(t)-\int_\Omega U^2_0$ of numerical solutions using a continuous Galerkin 
           scheme for the linear test problem \eqref{eq:linear_advec_two}.}
  \label{entrokey}
  \begin{tabular*}{\linewidth}{@{\extracolsep{\fill}}*3c@{}}
    \toprule
    Time & Correction \eqref{eq:correction} & Correction \eqref{eq:operator-entropy-correction-RD}\\
    \midrule
    0.50   &    $2.2904450548220892\cdot 10^{-19}  $  & $2.2904450546226423\cdot10^{-19} $ \\
    1.00    &    $1.2090807569861689\cdot 10^{-17}$    & $1.2090807582866649\cdot 10^{-17}$\\
    \bottomrule
  \end{tabular*}
\end{table}

The differences in Table~\ref{entrokey} are very small and the correction terms
lead in both cases to good results. 
Finally, the errors of the numerical solutions are nearly identical
for both correction terms.
Using \eqref{eq:correction}, the $L_{\infty}$ error is given by $ 1.5323772081379895\cdot 10^{-4}$.
In the other case, one obtains $  1.5323772103428853\cdot 10^{-4}$.
They intersect up to the power $10^{-12}$. 
If we increase the number of DOFs, the results are getting similar and
indistinguishable for this test case, in accordance with expectations
for a smooth linear problem.
Next, a nonlinear equation is considered.

\subsubsection*{Burgers' Type of Equation}
The problem is given by
\begin{equation}\label{eq:burgers_equation_like}
\begin{aligned}
      \partial_t u(t,x,y) +\partial_x(\cos u(t,x,y)) + \partial_y u(t,x,y)=0&,
      && (x,y) \in D, t\in (0, 0.2) \\
     u(0,x,y)=u_0(x,y)=\exp\bigl( -40(x^2+y^2) \bigr)&,&& (x,y)\in D,
\end{aligned}
\end{equation}
where $D$ is the unit disk in $\R^2$.
Again, outflow boundary conditions are considered and time integration is done
via SSPRK(3,3) with CFL number 0.1. A pure continuous Galerkin scheme of third order
with Bernstein polynomials is used in space. In this
test case, a small bump located around zero  moves up to the left and a shock
will appear after a finite time.
However, for this study, the time is considered before the shock appears.
The choice of $\cos$ in \eqref{eq:burgers_equation_like} is done on purpose
to guarantee that the integration with a quadrature rule will never be exact and in addition 
the flux is not convex.

%
 For the
square entropy  $U=u^2 / 2$, the differences $\int_\Omega U^2_{corr}(0.2)-\int_\Omega U^2_0$
of the entropies with the different corrections are
\begin{itemize}
  \item  $1.6819336368255455\cdot 10^{-15}$ for the correction \eqref{eq:correction} and
  \item  $1.6822445661098047\cdot 10^{-15}$ for the correction \eqref{eq:operator-entropy-correction-RD}
\end{itemize}
at $t = 0.2$.
The difference between the influence of the correction terms is also very small
for this test case  and will further decrease if the number of DOFs are increased. 
For smooth solution, no significant differences between the
two correction terms has been seen so far for different types of schemes and
other test cases not shown here. However, the advantages of one of the correction
terms compared to the other one cannot be excluded.

\subsection{Euler Equations}
This subsection contains two parts. First, the entropy correction term will be extended to the compressible Euler equations and compared. In the second part, a
first comparison between the flux-splitting approach and the application of corrections terms in classical nodal schemes is made. Here, everything is considered in the DG setting. Since
there exists a close connection between  FD  and DG using SBP operators  \cite{gassner2013skew}, the same holds true for SBP-FD schemes and can be found in the Appendix \ref{subsect:SBP_SAT_FD}. Finally, the schemes are considered on a tensor
structured grid and fulfill the SBP property, but triangular grids are also possible
\cite{chen2017entropy}. Nevertheless, we would like to stress that these schemes
fulfill several structural properties (SBP property) and are developed specifically
for experiments like these.

\subsubsection{Two-Dimensional  Sod and Shu-Osher Problem}
Here, we apply DeC together with either CG or the $\emph{psi}$-scheme using
Bernstein polynomials with the same order of accuracy, resulting in an
explicit space-time FE method.
In the first test, the focus will be on the classical two-dimensional Sod problem. 
The initial condition contains a jump in density and pressure,
\begin{equation*}
 (\rho_0, v_{x,0},v_{y,0}, p_0) =\begin{cases}
 (1,0,0,1),  \quad &0\leq r \leq 0.5,\\
 (0.125,0,0,0.1),  \quad &0.5<r \leq 1.\\
 \end{cases}
\end{equation*}
For the first simulation in Figure \ref{fig:initit_Sod}, the CFL number is set to $0.25$,  the test is run until $t= 0.25$,
and the classical correction term in the semidiscrete setting is applied together with a continuous CG  with jump stabilization, i.e. the correction is only used for  the space residual $\Phi_{\sigma,x}^K$ in \eqref{oneline}. Outflow boundary conditions are considered.
Without correction term, the CG scheme is breaking down at $t \approx 0.16$,
while it runs until the end when the correction term is used.

\begin{figure}[!ht]
\centering
    \includegraphics[width=0.32\textwidth]{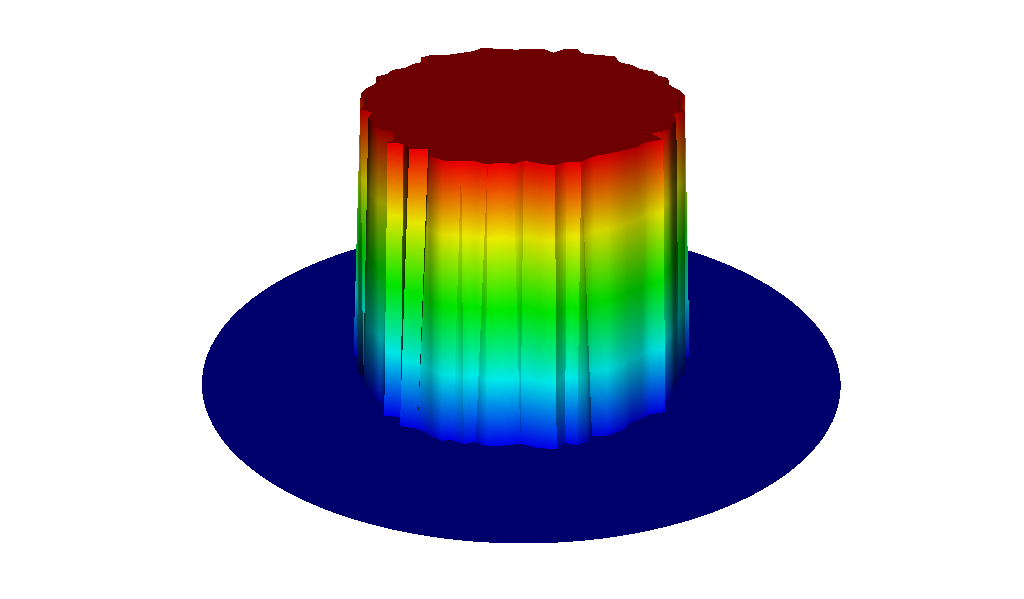}
        \includegraphics[width=0.32\textwidth]{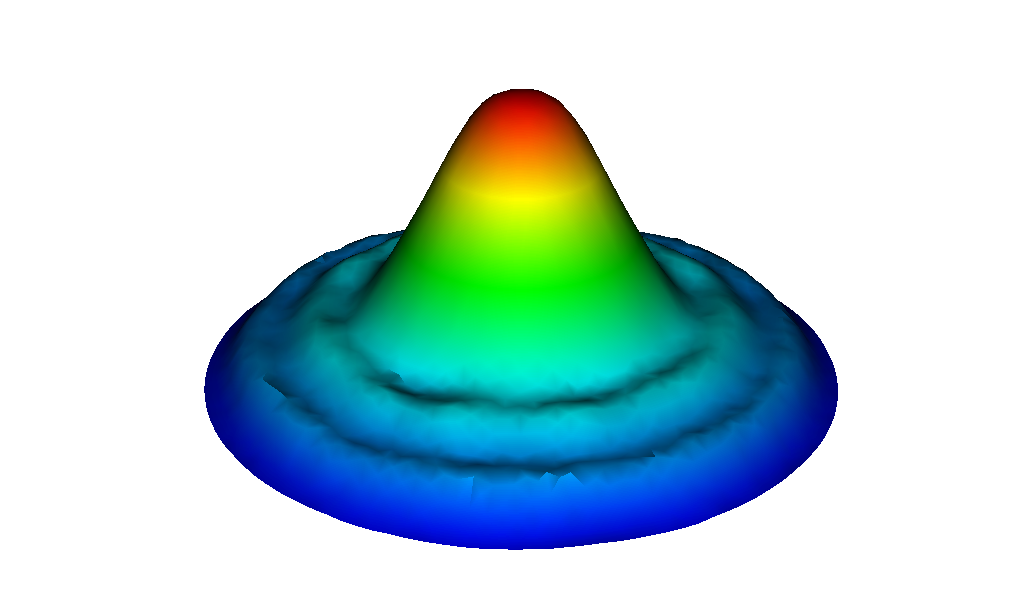}
    \includegraphics[width=0.32\textwidth]{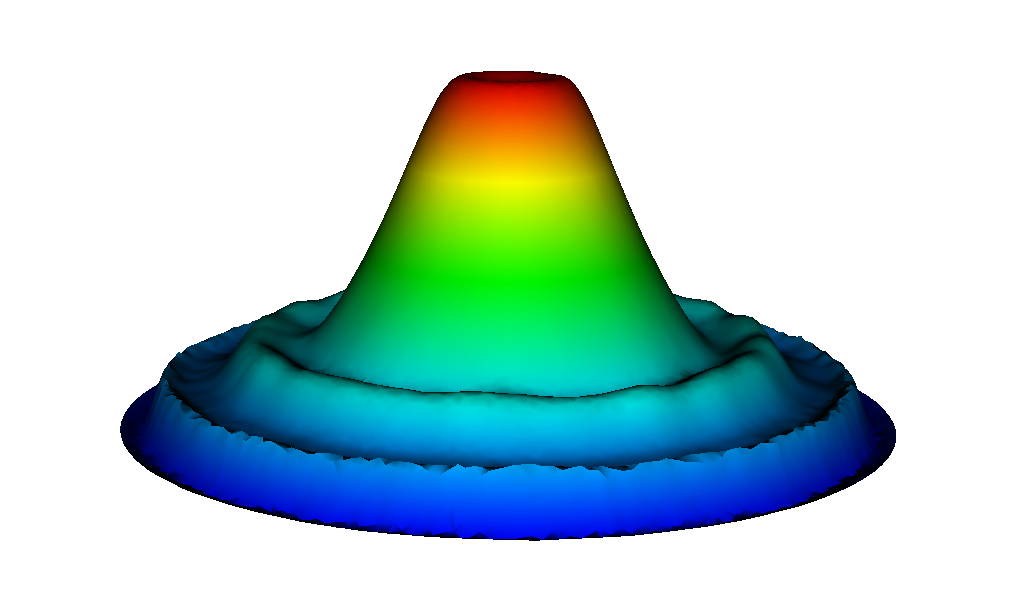}
  \caption{$t=0.25$, DeC(2,2), B1, residual correction, $\rho$-plot, 3576/13548 elements.}
 \label{fig:initit_Sod}
\end{figure}

In Figure~\ref{fig:both_correction_Sod}, both correction terms are applied. Without zooming the results are  indistinguishable. In the right picture, the bold line is (below) is the numerical solution using the classical correction term  \eqref{eq:operator-entropy-correction-RD} and the thin line is the result using the correction \eqref{eq:correction}.
This examples demonstrates well the improvement one obtains using entropy correction terms since the not corrected scheme is breaking down. 

\begin{figure}[!ht]
\centering
    \includegraphics[width=0.49\textwidth]{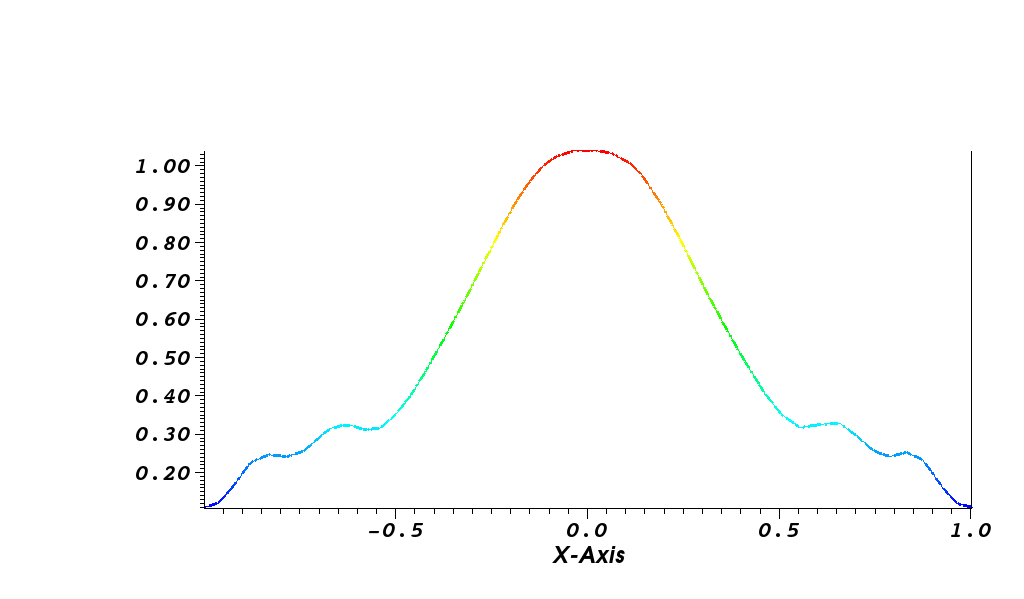}
    \includegraphics[width=0.49\textwidth]{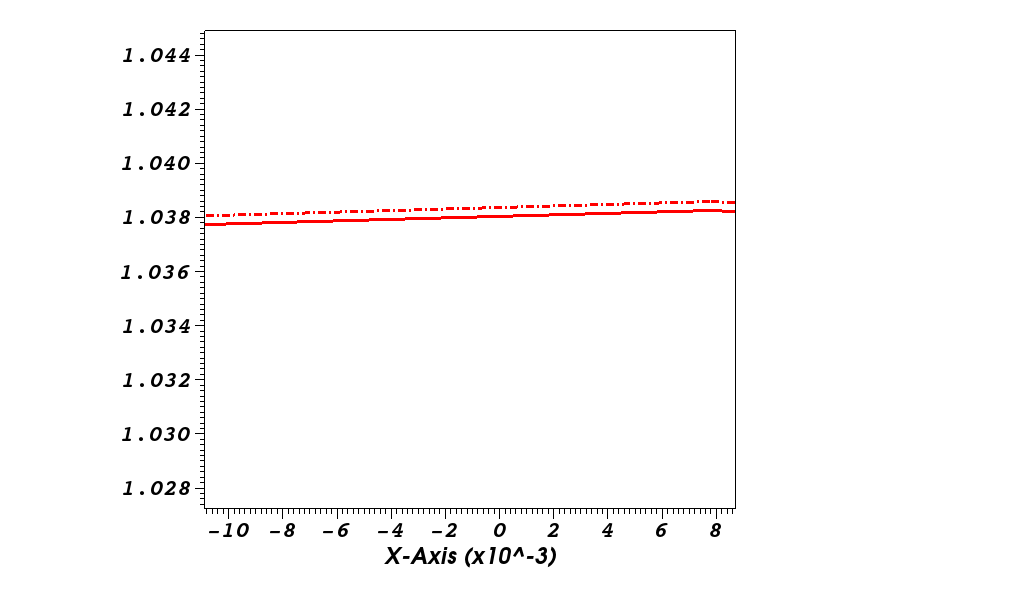}
  \caption{$t=0.25$, DeC(2,2), B1,  both correction terms,  $\rho$-plot, 3576 elements.}
  \label{fig:both_correction_Sod}
\end{figure}

Up to this point, only SEC/D schemes have been presented.
In Figure~\ref{fig:fully_correction_Sod},  the same test is considered with wall boundary conditions and the correction term is applied to the fully discrete update step, i.e.\ the complete bracket in \eqref{oneline} is corrected, resulting in  fully discrete EC
schemes. As baseline schemes the continuous Galerkin (straight line) and $\emph{psi}$ schemes (dotted) have been used.
One can recognize that the more dissipative character of classical $\emph{psi}$ scheme  has already some positive effect on the behavior of the approximation, although the numerical solutions are EC by construction. This can seen by zooming in since the plateaus are no longer flat using the Galerkin scheme.  The reason for this lies in the imposed wall boundary conditions and a further
investigation of this behavior will be addressed in future work.

\begin{figure}[!ht]
\centering
    \includegraphics[width=0.55\textwidth]{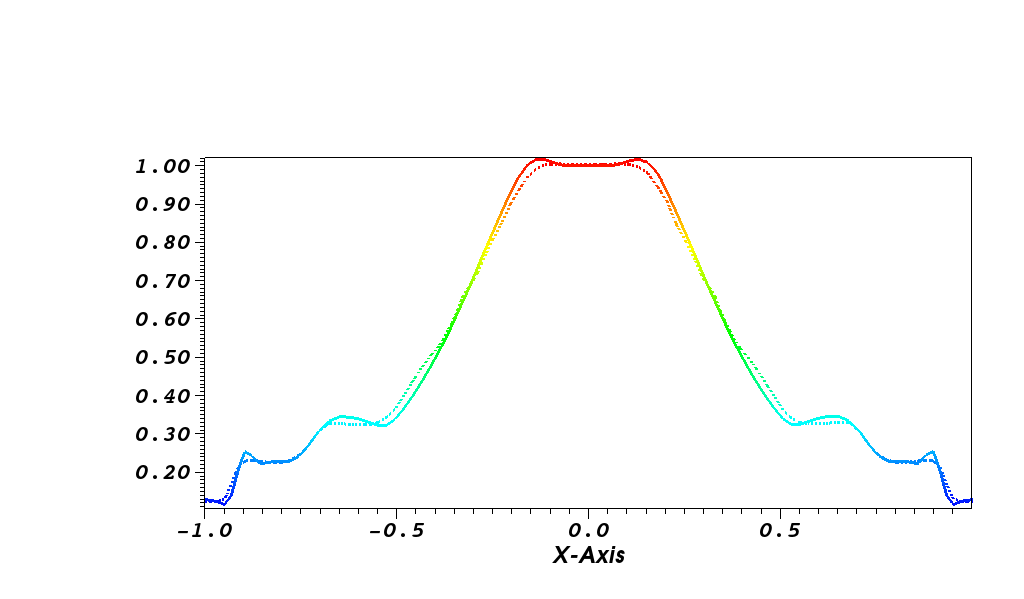}
    \includegraphics[width=0.4\textwidth]{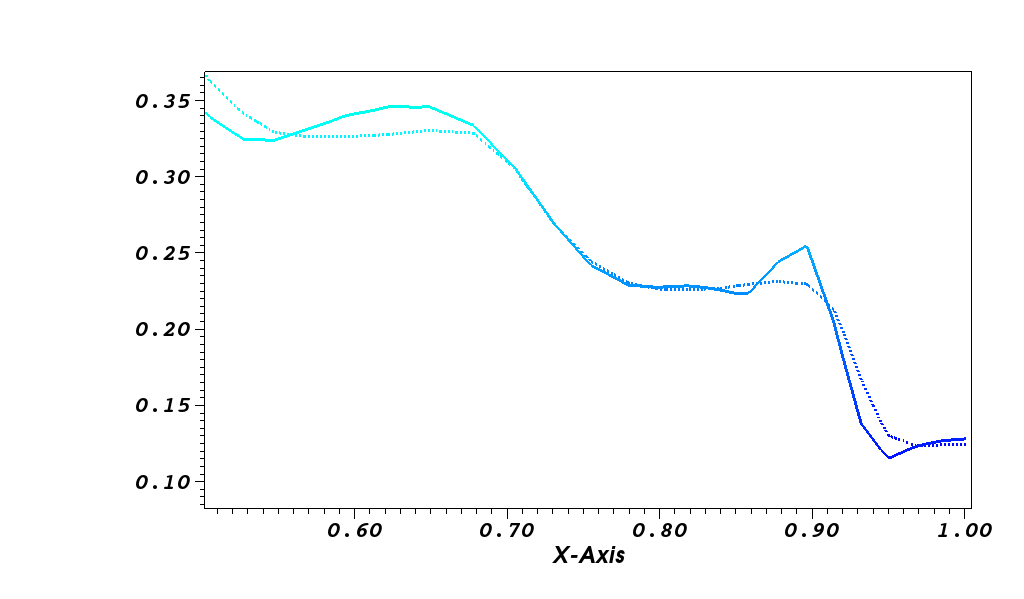}
  \caption{$t=0.25$, DeC(2,2), B1,  residual correction,  $\rho$-plot, 8576 elements.}
  \label{fig:fully_correction_Sod}
\end{figure}

The semidiscrete (dotted) and the fully discrete (solid) Galerkin schemes
are compared in Figure~\ref{fig:fully_semi_correction_Sod}.
Small differences can be recognized, especially around the shocks. However, further investigations and a comparison with the relaxation approach will be part of future work.
The correction term \eqref{eq:operator-entropy-correction-RD} is mainly applied. However, analogous results in terms of the quantitative behavior are obtained if the correction term \eqref{eq:correction} would have been used.
\begin{figure}[!ht]
\centering
    \includegraphics[width=0.54\textwidth]{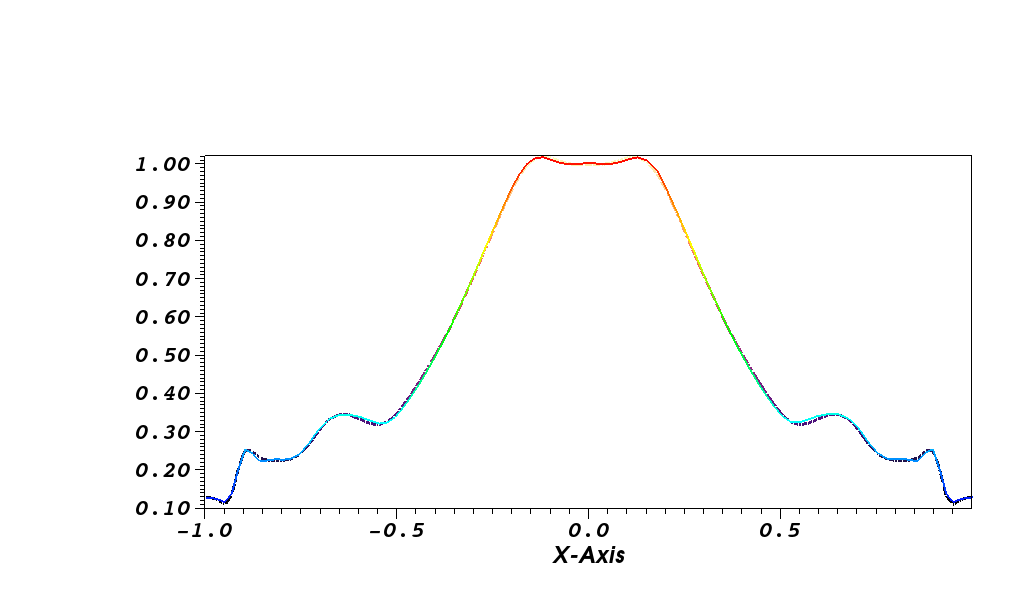}
    \includegraphics[width=0.44\textwidth]{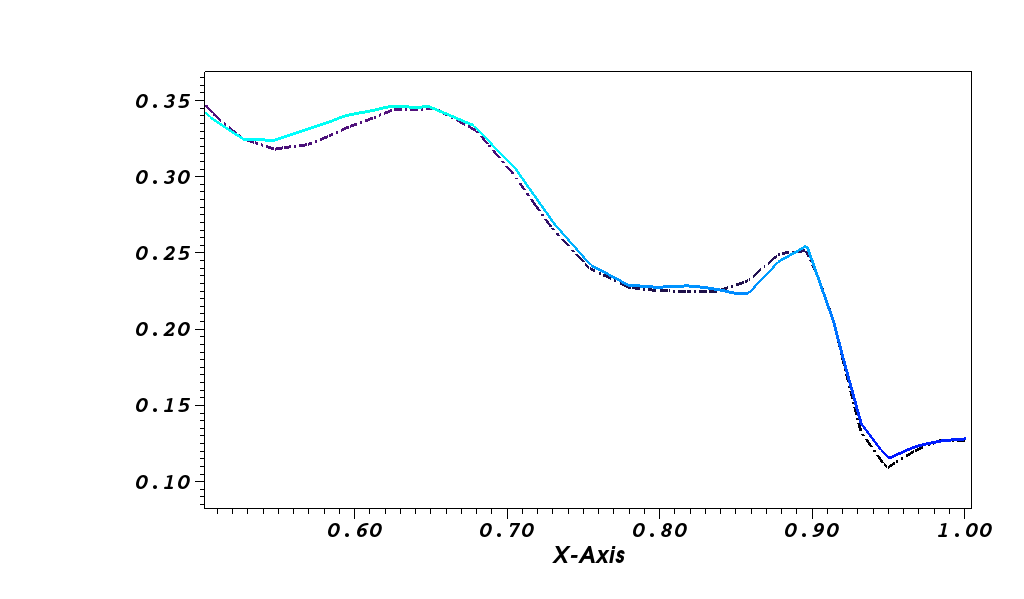}
  \caption{$t=0.25$, DeC(2,2), B1,  residual correction,  $\rho$-plot, 8576 elements.}
  \label{fig:fully_semi_correction_Sod}
\end{figure}
Finally, extensions to higher polynomial degrees are also possible if one can guarantee that both the pressure as well as the density remain positive. This is important for the application of the entropy correction term, not only because of physical reasons but also because of the switch between conservative and entropy variables.
In Figure~\ref{fig:B3}, the corrected Galerkin methods are applied and positivity of density and pressure is ensured at all DOFs through the MOOD procedure \cite{bacigaluppi2019posteriori}. However, other limiting strategies can be used as well, e.g.\ \cite{kuzmin2020entropy}.
\begin{figure}[!ht]
\centering
    \includegraphics[width=0.4\textwidth]{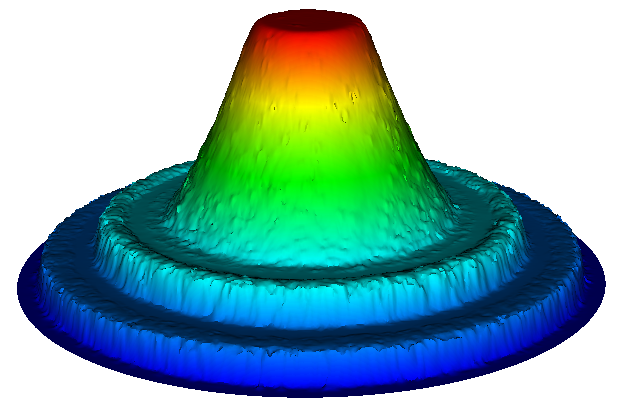}
    \includegraphics[width=0.44\textwidth]{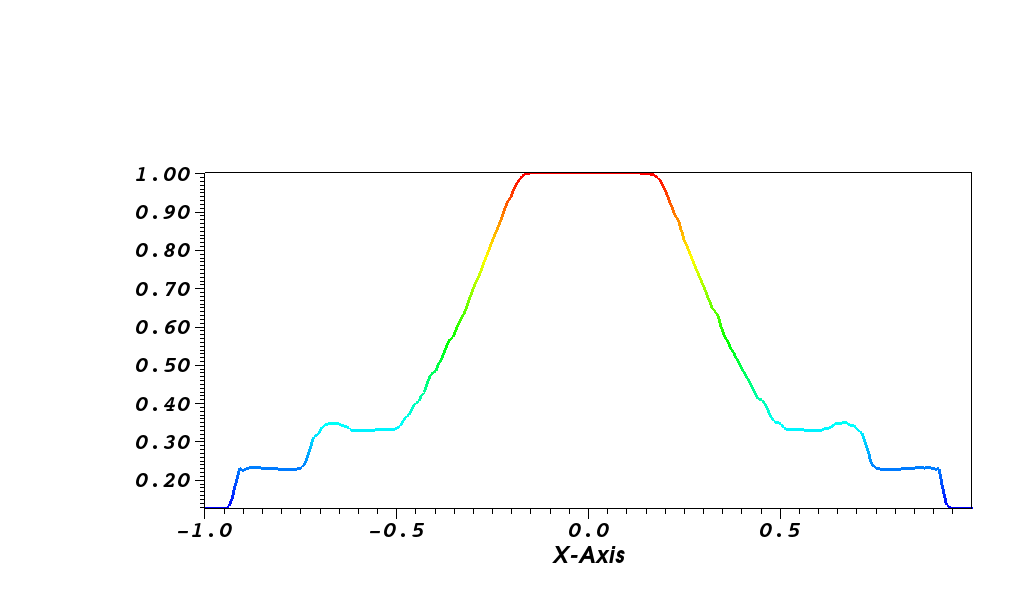}
  \caption{  $t=0.25$, DeC(4,4), B3,  residual correction,  $\rho$-plot, 13548 elements }
  \label{fig:B3}
\end{figure}
A second test is considered for completeness. Here, the classical 1D Shu-Osher test is extended to two dimension  with radial speeds and initial conditions
 \begin{equation*}
 (\rho_0, v_{x,0},v_{y,0}, p_0)
 =
 \begin{cases}
 (3.857143,\sqrt{2.629\rho_0}\frac{x}{r},\sqrt{2.629\rho_0}\frac{y}{r},10.\overline{3}),  \quad &0\leq r \leq 1,\\
 (1+0.5\sin (5r),0,0,1),  \quad &1<r <4, \\
 (1+0.5\sin (20),0,0,1),  \quad &4\leq r. \\
 \end{cases}
 \end{equation*}
 Figure \ref{Fig:Shu} shows the initial conditions, an intermediate result after 150 steps, and the final result
 at $t=1.8$ (373 steps) using the FEC Galerkin method. The SEC Galerkin method is nearly indistinguishable from this result due to the properties of the DeC approach \cite{abgrall2017high, abgrall2021relaxation}
 and the usage of high order quadrature rules.

\begin{figure}[!ht]
\centering
 \includegraphics[width=0.3\textwidth]{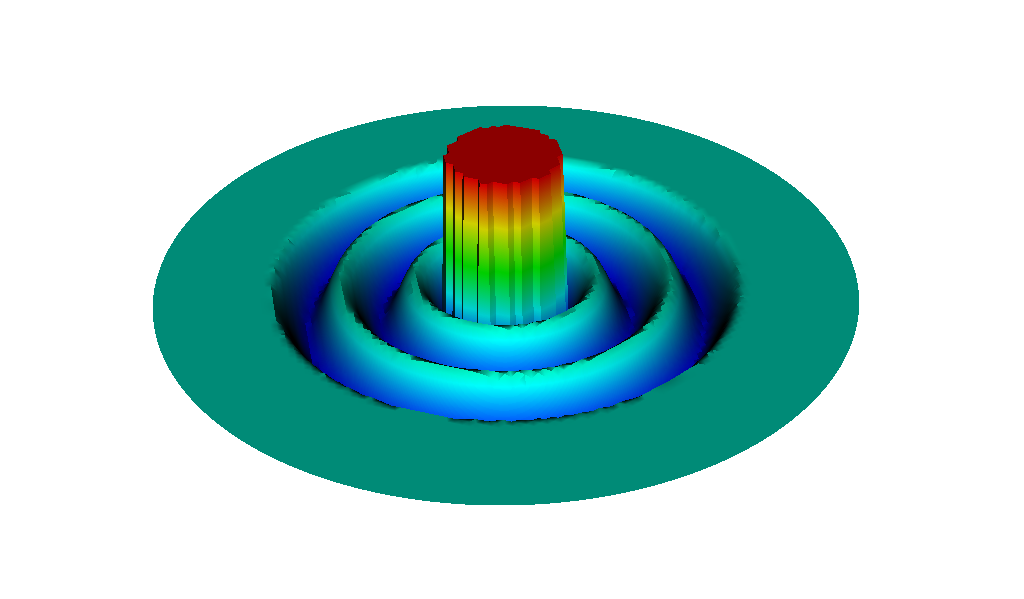}
    \includegraphics[width=0.32\textwidth]{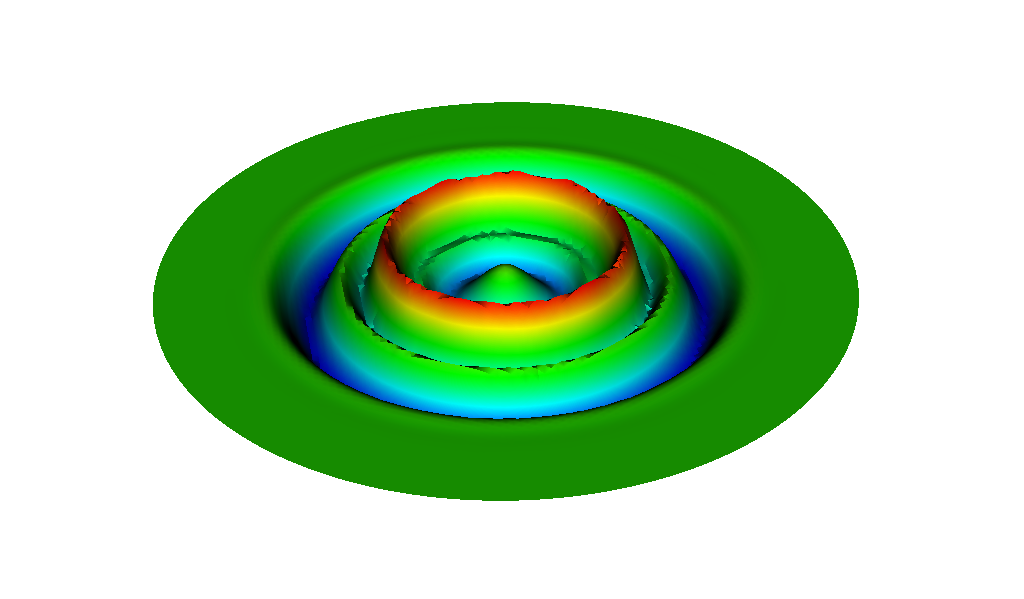}
    \includegraphics[width=0.32\textwidth]{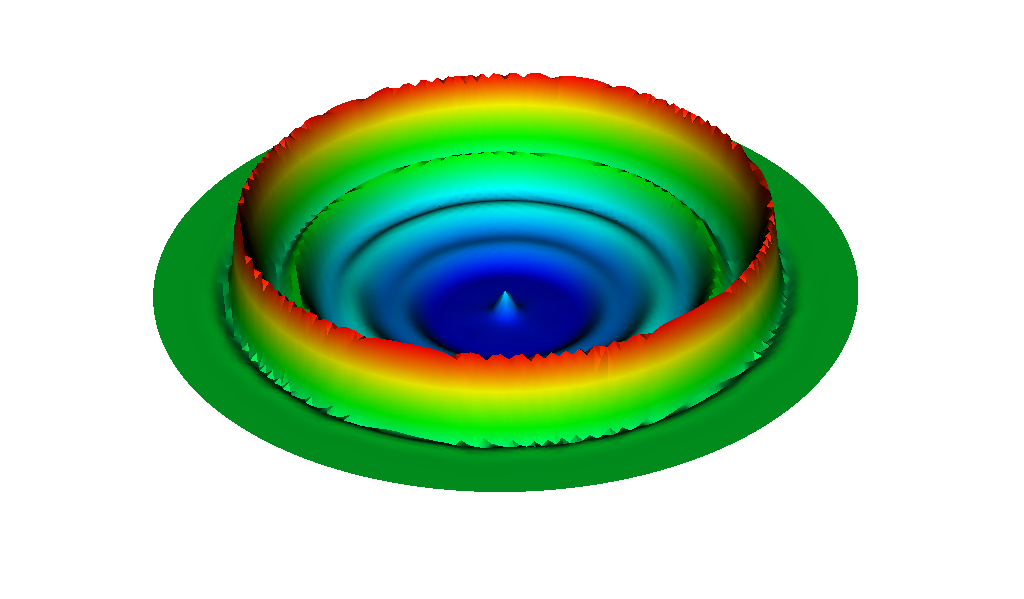}
  \caption{ Initial condition, (150 steps) and $t=1.8$ (373 steps), DeC(2,2), B1, classical correction, $\rho$-plot.}
  \label{Fig:Shu}
\end{figure}

\subsubsection{SBP-SAT-DG Setting}\label{subsection_SBP_DG}

\begin{figure}
\centering
  \includegraphics[width=\textwidth]{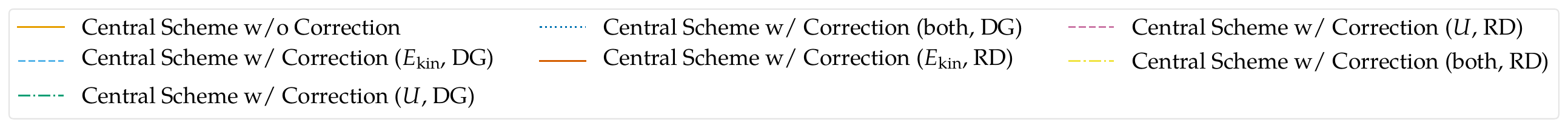}
  \\
  \begin{subfigure}{0.49\textwidth}
    \centering
    \includegraphics[width=\textwidth]{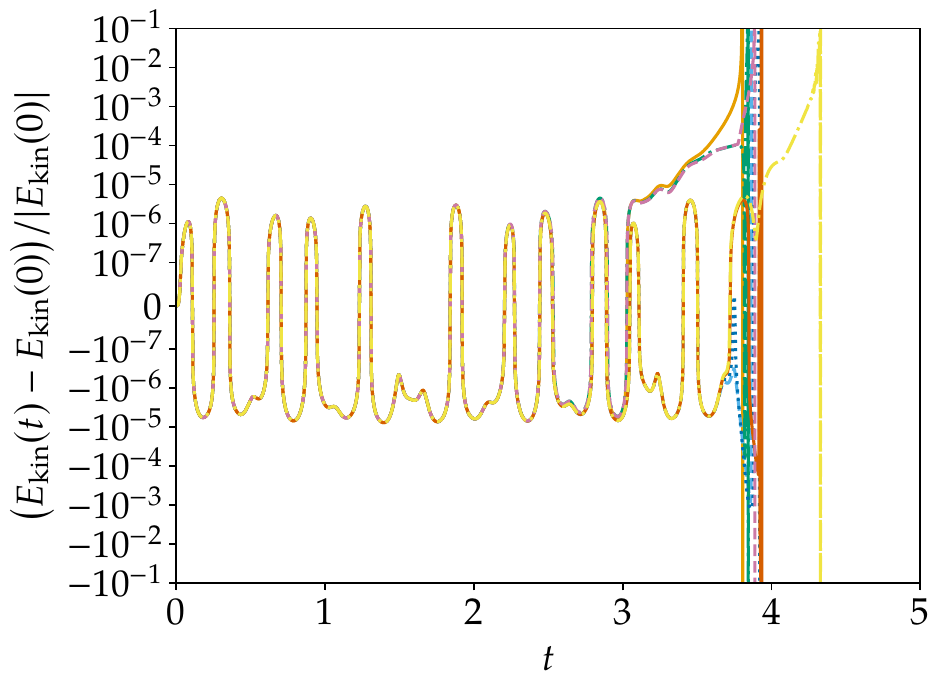}
    \caption{Kinetic Energy.}
  \end{subfigure}%
  \begin{subfigure}{0.49\textwidth}
    \centering
    \includegraphics[width=\textwidth]{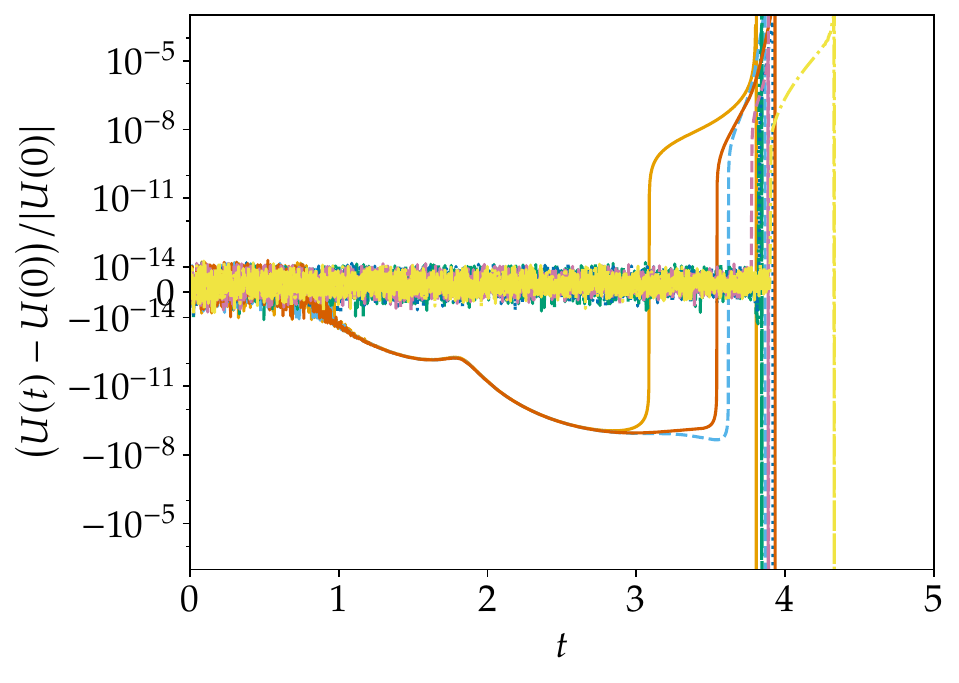}
    \caption{Entropy.}
  \end{subfigure}%
  \caption{Relative kinetic energy $\Ekin$ and entropy $U$ of numerical solutions
           of the compressible Euler equations with a Taylor-Green 
     }
  \label{fig:Euler_DG_TGV}
\end{figure}

In the second part,  a first comparison
between the flux differencing approach and the application of
correction terms in classical nodal schemes is made.
We consider the compressible Euler equations
 together with Taylor-Green vortex initial conditions given by
\begin{equation*}
\begin{aligned}
  \rho(0,x,y) &= 1,
  &
  v_x(0,x,y) &= \sin(x) \cos(y),
  \\
  v_y(0,x,y) &= -\cos(x) \sin(y),
  &
  p(0,x,y) &= \frac{100}{\gamma} + \frac{ \cos(2x) + \cos(2y) }{4},
\end{aligned}
\end{equation*}
for $(x,y) \in [0,2\pi]^2$ and periodic boundary conditions.
Then, a nodal DG scheme on Lobatto Legendre nodes
and the associated Lobatto Legendre quadrature
\cite[Eq.~25.4.32]{abramowitz1972handbook} for the mass matrix with
polynomial degree $p = 5$ and $16$ tensor product quadrilateral elements per
coordinate direction is  used. The numerical flux at the boundaries is
the one of \cite[Theorem~7.8]{ranocha2018thesis}
and the classical central scheme is used in the interior of each element.
The numerical solutions have been integrated in time with the fourth order, ten-stage,
SSPRK method of \cite{ketcheson2008highly} and
a constant time step $\Delta t = \frac{\Delta x}{10} \frac{1}{p^2 + 1}$, where
$\Delta x$ is the width of one element.
Using the central scheme in the interior of each element results
in a blow-up  at
$t \approx 3.8$, cf.\ Figure \ref{fig:Euler_DG_TGV}.
Applying a correction for $\Ekin$ removes the increase of the
kinetic energy before the blow-up and applying a correction for $U$ yields a nearly
constant entropy (before the blow-up).
As before, there is no big difference between the basic choices of the correction terms
\eqref{eq:correction} and \eqref{eq:operator-entropy-correction}.
The different weighting does not seem to be crucial for these unstable calculations
here. The only minor exception is the correction term for both $\Ekin$ and $U$
with the  weighting  \eqref{eq:correction}. It blows up slightly later at
$t \approx 4.3$.
The computations using flux differencing schemes in the interior of each element
 remain stable and do not blow up until $t\approx 35$ is reached, in contrast to the central schemes with correction terms which breaks before.
The main purpose of the flux-differencing approach is to split the volume term in the discretization
(resulting in a split or skew-symmetric formulation). 
Other split formulations for the volume term can be found in the literature, cf.\ \cite{ranocha2018thesis} and references therein. Another one is based on the flux of Ducros et al. \cite{ducros2000high},
resulting in a DG scheme which is neither entropy stable nor KEP.
However, by applying Ducros flux, a splitting for the volume term is obtained and applying 
to this scheme now the correction terms, one gets the desired results which are presented 
in Figure~\ref{fig:Euler_DG_TGV__Du}.  Using for instance both the energy and entropy correction,
the scheme is similar to the flux differencing scheme using Ranocha's flux.
 In particular, the entropy can be conserved and the
 numerical solutions does not blow up during the computation.
 Here, we used this numerical flux in the surface term again only for comparison.
The numerical flux in the surface integral does not have a major influence on the stability when the correction terms are applied as well.
We get indistinguishable results for Ducros' flux and others in the surface integrals.
 To clarify, if we apply a proper splitting technique for the volume integral in the DG formulation
 and use the corrections, we obtain results equivalent to the flux difference approach using these highly sophisticated numerical fluxes.
But what is the problem with the classical DG scheme then? Actually, by a closer analysis, it can be recognized that when the DG scheme with entropy corrections results in negative densities and pressures at some nodal values, then neither the change to entropy variables nor the usage of the entropy correction is possible anymore. Thus, the test crashes.
 The splitting of the volume terms prevents this. An alternative would be the use of positivity preserving limiters or positivity preserving time-integration schemes, which should and will be part of future research. However, the observation made in this subsection can be summarized as follows:
 \begin{itemize}
  \item If the structural properties (SBP) are given, the
  flux difference approach using the flux of
  \cite{ranocha2018comparison, ranocha2018thesis, ranocha2021preventing}
  should be applied since
  it yields an EC/D and KEP scheme with some advantages compared to other fluxes.
  \item If one prefers other splitting techniques the respective 
  corrections terms will always lead to the desired properties and comparable results 
  to the splitting using Ranocha's flux. This technique is universally applicable.
  \item If the structural properties are not given, the 
  correction terms can always be applied and guarantees the desired properties.
  It will also increase the stability of the scheme.
  \item If the physical constraints (i.e. positivity of the density and pressure) are ensured, 
 the entropy correction schemes together with a ``proper'' space discretization are at least as good as state-of-the-art schemes.
 \end{itemize}
It is clear that this is only a first comparison. Further experiments have to be conducted, including
for example limiters to avoid the negativity of the density and pressure as done before and seen in  Figure~\ref{fig:B3}. Further, 
also numerical errors and cancellations will be taken more into account. Both are part of future research. 
Finally, it should be stressed that the correction term is universal whereas
the numerical fluxes have to be constructed specifically for each equation.
If other problems are considered, one has to work on the construction 
of numerical fluxes again which have the requested properties. 
Therefore, more efforts have to been made and 
it is unknown if this always yields a result. Here, the presented correction term is an enrichment. In Section~\ref{sec:remark}, this topic will be revisited and some
examples motivating our formulation formulation will be discussed.
\begin{figure}
\centering
  \includegraphics[width=\textwidth]{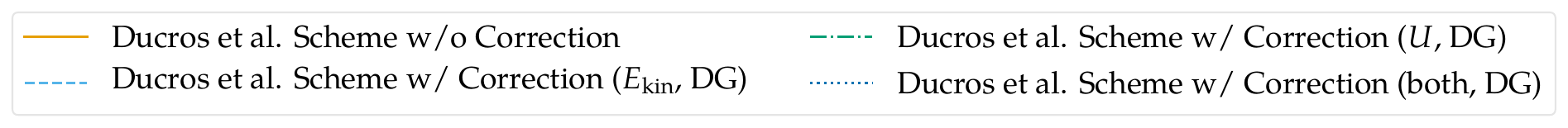}
  \\
  \begin{subfigure}{0.46\textwidth}
    \centering
    \includegraphics[width=\textwidth]{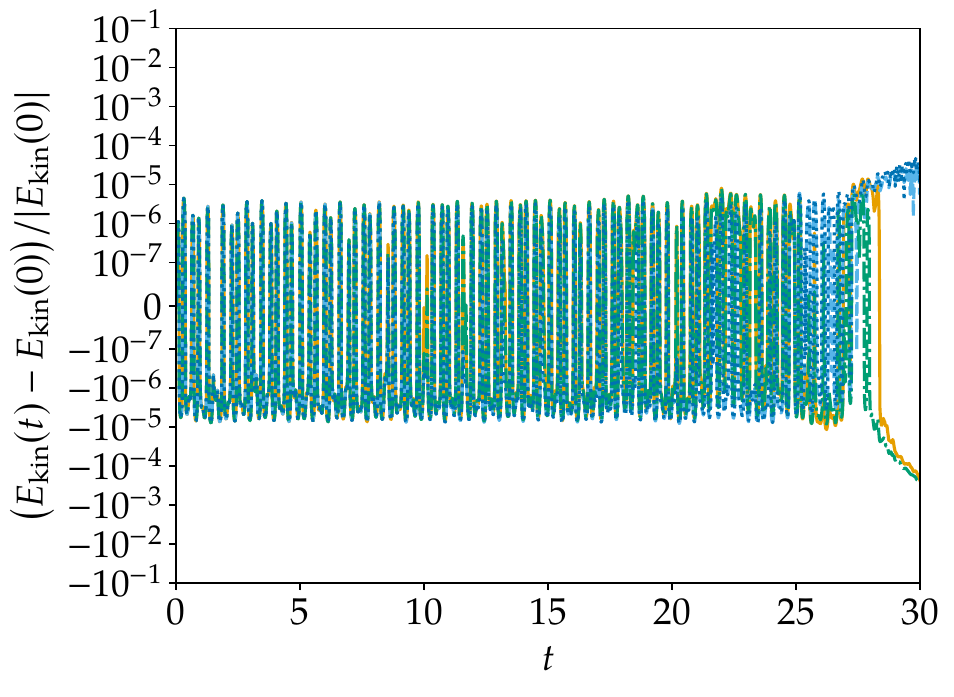}
    \caption{Kinetic Energy.}
  \end{subfigure}%
  \hspace{\fill}
  \begin{subfigure}{0.46\textwidth}
    \centering
    \includegraphics[width=\textwidth]{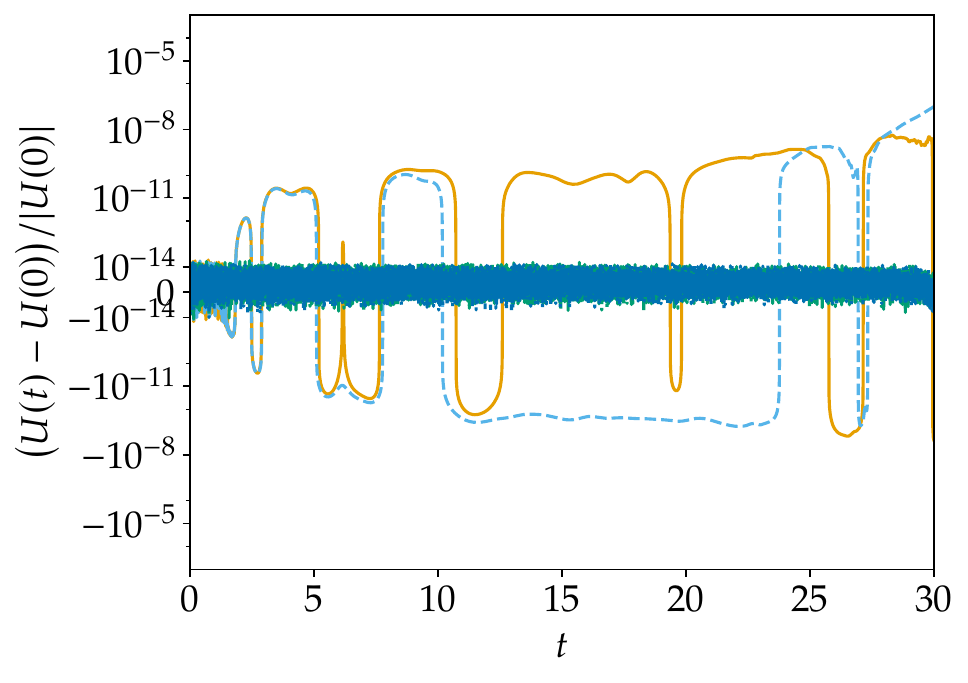}
    \caption{Entropy.}
  \end{subfigure}%
  \caption{Relative kinetic energy $\Ekin$ and entropy $U$ of numerical
           solutions of the compressible Euler equations with a Taylor-Green
           vortex initial condition. The nodal discontinuous Galerkin schemes
           use polynomials of degree $p=5$ on Lobatto Legendre nodes
           \cite[Eq.~25.4.14]{abramowitz1972handbook}, the flux difference
           split form with the flux of Ducros et al. \cite{ducros2000high,
           gassner2016split} in the interior of each element, and the
           numerical fluxes of \cite[Theorem~7.8]{ranocha2018thesis}
           between elements.}
  \label{fig:Euler_DG_TGV__Du}
\end{figure}

\subsection{Convergence Test}
Here, a convergence test using the initial condition
\begin{equation*}
\label{eq:euler-lin-adv}
\begin{aligned}
  \rho(0,x,y) &= 1 + \frac{1}{2} \sin(\pi x),
  &
  v_x(0,x,y) &= 1,
  \\
  v_y(0,x,y) &= 0,
  &
  p(0,x,y) &= 1,
\end{aligned}
\end{equation*}
for the Euler equations \eqref{eq:Euler} in the periodic domain $[0,2] \times [0,1]$
using DG schemes is conducted using $(N_x,N_y=1)$ elements in $x$-$y$ direction and polynomials of degree $p = 4$.
The error of the density at the final time $t = 6$ is computed using the Lobatto
Legendre quadrature.
As can be seen in Table~\ref{tab:euler_DG_convergence}, there are no significant
differences between the baseline central scheme and the ones applying a correction
term, in accordance with results of \cite{abgrall2018general}. For low resolutions,
the correction term \eqref{eq:operator-entropy-correction-RD} using a weighting
by the mass matrix yields slightly lower errors than the one proposed originally
in \cite{abgrall2018general}.
Some EOCs are a bit larger than the theoretical EOC of $p+1$. This might
be caused by superconvergence effects for DG methods, since the solution returns
to the initial condition and the errors are measured at nodal points, at which
the solution may observe superconvergence.

\begin{table}[!ht]
\centering
  \caption{Convergence rates for central and corrected DG schemes using polynomials
           of degree $p = 4$ for the Euler equations with initial condition
           \eqref{eq:euler-lin-adv}.}
  \label{tab:euler_DG_convergence}
  \begin{tabular*}{\linewidth}{@{\extracolsep{\fill}}*7c@{}}
    \toprule
    & \multicolumn{2}{c}{Central Scheme}
    & \multicolumn{2}{c}{Correction \eqref{eq:operator-entropy-correction-RD}}
    & \multicolumn{2}{c}{Correction \eqref{eq:correction}}
    \\
    $N_x$
    & $\norm{\rho - \rho_0}_M$ & EOC & $\norm{\rho - \rho_0}_M$ & EOC
    & $\norm{\rho - \rho_0}_M$ & EOC
    \\
    \midrule
     5 & 1.312e-05 &       & 1.025e-04 &       & 2.537e-04 \\
    10 & 1.394e-06 & +3.23 & 1.994e-06 & +5.68 & 3.041e-06 & +6.38 \\
    15 & 3.095e-07 & +3.71 & 3.217e-07 & +4.50 & 3.436e-07 & +5.38 \\
    20 & 6.426e-08 & +5.46 & 6.486e-08 & +5.57 & 6.714e-08 & +5.68 \\
    25 & 1.810e-08 & +5.68 & 1.819e-08 & +5.70 & 1.836e-08 & +5.81 \\
    \bottomrule
  \end{tabular*}
\end{table}

\section{Why a New Formulation? Some Motivating Examples}
\label{sec:remark}
In order to illustrate what we were explaining in the introduction, let us review shortly how entropy preserving schemes are classically obtained. An entropy conservative numerical flux $\fnum$ should satisfy Tadmor's condition:
\begin{equation}
\label{remi:tadmor}\jump{v} \fnum =\jump{\psi}
\end{equation} where $v$ is the entropy variable, ( $v=\nabla_uS$) and $\psi$ is the potential,
$$g=v^T f-\psi$$ where $g$ is the entropy flux.

For systems, $v$ and $ \fnum$ are vectors, so that \eqref{remi:tadmor} must be obtained in a similar way as the Roe average is obtained. This is where complications starts because one need to compute $\psi$ first. In the case of the Euler equations, $S=-\rho s$ where $s$ is the specific entropy, and since
$v=\nabla_{(\rho, \rho \mathbf{u}, E)} S$, we need an expression of $s$ in term of the conserved variables, or some other set of variables so that the mapping is one-to-one. For inert gases, the specific entropy is defined from Gibb's relation:
\begin{equation}\label{remi:gibbs}
T\; ds=d\varepsilon -\frac{p}{\rho^2}d\rho,
\end{equation}
where $T$ is the absolute temperature, $\varepsilon$ the specific heat.

What makes the things relatively easy for the calorically perfect gases case is that specific heat is a constant, as well as $\gamma=\tfrac{c_p}{c_v}$, so that $\varepsilon=c_v T$, and the specific entropy has a simple form
$$s= \log p-\gamma \log \rho$$ thanks to Mayer's relation
$c_p-c_v=\tfrac{\mathcal{R}}{M}$, $\mathcal{R}$ being Avogadro's number and $M$ the molar mass.

For a perfect gas, but a non-calorically perfect one, $\varepsilon=\int_{T_0}^T c_v(T) dT$, and since $p=\rho \frac{\mathcal{R}}{M}T$, we obtain
$$s=\int_{T_0}^T \frac{c_v(T)}{T} dT-\frac{\mathcal{R}}{M}\log(\rho)+s_0,$$ and the situation becomes way more complicated in general, and very dependent  on the structure of the function $T\mapsto c_v(T)$. It can be analytical (such as in the case of diatomic molecules in air) or tabulated (such as in the NIST-Janaf tables \cite{janaf}). In the analytical case, it is very likely possible to determine analytically the correct average, such as what has been done for the Roe average and for a complex equation of state, cf. \cite{abgrallEOS} for one example. All of this boils down to identifying the correct variables that make the algebra a bit magic thanks to differentiation rules such as
$$\Delta fg=\underline{f}\Delta g +\overline{g}\Delta f, \quad \Delta h=h_L-h_R,$$
where, for a "well-chosen" parameter $\lambda\in [0,1]$,
$$\overline{f}=\lambda f_R+(1-\lambda) f_L, \qquad \underline{f}=\lambda f_L+(1-\lambda) f_R.$$
In the case of the Roe average, the magic parameter is
$\lambda=\dfrac{\sqrt{\rho_L}}{\sqrt{\rho_R}+\sqrt{\rho_L}}$ and the algebra is much simplified by
$$\underline{\rho f}=\underline{\rho}\overline{f}, \qquad \underline{f}=\sqrt{\rho_L\rho_R}.$$

Unfortunately, this magic is very much case dependent, and can be very complicated to manage. This why any way of avoiding this tricky algebra for complicated forms of the entropy that a priori needs to have some analytical form is welcomed. Further note  that in most cases, this analytical form is not very clean or depends on a look-up table. This is only one example. The situation can become even worse if one asks kinetic energy preservation, and/or additional constraints, such as local preservation of the kinetic momentum.

\section{Summary and Conclusions}
\label{sec:summary}

In this paper, a reinterpretation and extension of entropy correction terms
proposed in \cite{abgrall2018general} is given.
Based on a characterization of these terms as solutions of certain optimization
problems, different correction terms are determined by the choice of discrete
norms. Additionally, these terms are adapted to numerical methods such as
discontinuous Galerkin  schemes. 
In numerical simulations,
the various correction terms are tested both in semidiscrete and fully discrete methods and compared also to flux differencing
approaches \cite{fisher2013high,gassner2016split}. These tests demonstrate that
there is no significant difference between the two basic choices of correction
terms (given in \eqref{eq:correction} and \eqref{eq:correction-abgrall-as-operator},
respectively). All of these optimization problems are solved
analytically, resulting in the same runtime performance without the need
for an optimization solver.

Using a flux difference formulation shows advantages compared to the
application of corrections terms to a simple central scheme.
While the correction
terms work and preserve the kinetic energy and/or conserve the entropy for the
Euler equations as expected, they cannot prevent a blow-up of numerical solutions for a demanding
Taylor--Green vortex type initial condition because the corrections terms cannot work correctly due to the generation of negative pressure and densities. If this can be avoided by further techniques,
the correction terms can be
applied successfully, yielding an entropy conservative scheme that does not blow
up during the computation.

This can be interpreted as follows: There is no free lunch. At some point,
one has to put work into the methods. 
The usage and implementation of the entropy correction term is quite simple. However,
one has to ensure that physical constraints (positivity of density/pressure in this case) are satisfied to obtain the desired results.
On the other hand, the flux difference framework
includes some assumptions on the quadrature, i.e.\ the SBP property
as discrete analogue of integration by parts, and the grid structure.
Here, a lot of effort has already been invested, as one can recognize by the
immense literature in this field. Also the construction of proper two-points fluxes is challenging and
it is unclear whether good fluxes can be found for every system as mentioned before.
The application of the correction
terms does not need these assumptions and is consequently more general. In particular,
it can be applied to members of the framework of residual distribution schemes 
and therefore to nearly any finite-element 
or finite volume based scheme on any grid. 
It is a universal tool.

However, it should be stressed out that the
success of flux differencing schemes cannot be attributed solely to resulting
equations for the kinetic energy or entropy across elements. Inter-element/subcell
local equations (which hold for these schemes) might play a role for the improved
numerical stability as well.
These experiments are only a first comparison and further tests have to be
done in this direction but this is  not topic of this current paper.

Besides the reinterpretation of the correction terms, their extension to new
applications has been conducted. Here, novel generalizations to entropy
inequalities, multiple constraints, and kinetic energy preservation for the Euler
equations are developed and verified by numerical simulations with a focus on
multiple constraints such as conserving the entropy and preserving the kinetic
energy simultaneously. Simultaneously, an approach is presented to obtain FEC schemes using only the entropy correction terms.

\appendix
\section{Additional Examples}\label{sec:appendix}

\subsection{Applications to Grid Refinement and Coarsening}
\label{sec:grid-refinement}
There is a certain interest in spatial adaptivity for EC/D semidiscretizations
\cite{friedrich2018entropyHP}.
If space and time adaption has to be performed, grid refinement and
coarsening operators transferring the numerical approximation from one grid to
another have to be constructed. Ideally, these should respect the entropy
dissipative behavior.
Consider  fine   and  coarse grids $\Omega_f$, $\Omega_c$. In the following,
it is assumed that a nodal discontinuous element type basis is used and that
there are associated coarse/fine and fine/coarse  interpolation operators $\Ifc$ and $\Icf$.  Numerical solutions and
operators on the coarse/fine grid will be written using upper indices $c/f$.
While standard interpolation and $L^2$ projection operators conserve the total
mass for reasonable choices of the basis functions and quadrature rule, it is
in general impossible to obtain strict inequalities for entropies. Hence, it is
of interest to apply an optimization approach similar to the one described in
the previous section.
However, there are important differences concerning the optimization approach to
entropy dissipativity between computations of spatial semidiscretizations and
refinement/coarsening operators:
\begin{itemize}
  \item
  Computing the time derivative of the entropy results in a linearization of
  the problem, making it much easier. In fact, a closed solution is given in
  the previous sections. For the refinement/coarsening operators, such a closed
  form solution does not seem to be available in general.
  \item
  The spatial semidiscretization has to be evaluated for every element and every time
  step, possibly multiple times. In contrast, the refinement/coarsening operators
  will be evaluated significantly fewer times.
\end{itemize}
To sum up, entropy stable refinement/coarsening operators are more expensive but
also used less often. Hence, they can be of interest in applications requiring
strict entropy inequalities. 
\subsubsection*{Refinement}
Given a solution $u^c$ on the coarse grid, an optimization problem for the solution
$u^f$ on the fine grid is
\begin{equation}
\label{eq:optimization-refinement}
  \min_{u^f} \frac{1}{2} \norm{u^f - \Ifc u^c}_{M^f}^2
  \st 1^T M^f u^f = 1^T M^c u^c,\; 1^T M^f U(u^f) \leq 1^T M^c U(u^c).
\end{equation}
Since the norm $\norm{\cdot}_{M^f}$ is strictly convex and  coercive and the entropy
$U$ is convex, there exists a unique solution of \eqref{eq:optimization-refinement}.
This convex problem can be solved using standard constrained optimization algorithms.
For the numerical examples presented below, Ipopt \cite{wachter2006implementation}
has been used via the interface provided by JuMP \cite{dunning2017jump},
using the default options.
After some simplifications, the first order necessary condition becomes
\cite[Theorem~12.1]{nocedal1999numerical}
\begin{equation}
  u^f = \Ifc u^c + \lambda \left( w(u^f) - \frac{1^T M^f w(u^f)}{1^T M^f 1} 1 \right),
\end{equation}
where $\lambda \geq 0$ is a Lagrange multiplier. Assuming that $\Ifc u^c$ is
already a good approximation to $u^f$, $w(u^f)$ can be substituted by $w(\Ifc u^c)$,
resulting in
\begin{equation}
\label{eq:optimization-refinement-simplified}
\begin{gathered}
  \min_{\lambda} \frac{1}{2} \lambda^2
  \\
  \st 1^T M^f U\Biggl(
    \Ifc u^c + \lambda \left( w(\Ifc u^c) - \frac{1^T M^f w(\Ifc u^c)}{1^T M^f 1} 1 \right)
  \Biggr) \leq 1^T M^c U(u^c).
\end{gathered}
\end{equation}
Compared to \eqref{eq:optimization-refinement}, where the whole vector
$u^f$ is the unknown, \eqref{eq:optimization-refinement-simplified} is
a scalar optimization problem which can be solved more efficiently.
Indeed, it is basically a scalar root finding procedure: If the interpolation
$u^c \to \Ifc u^c$ is entropy dissipative, $\lambda = 0$ can be chosen.
Otherwise, $\gamma$ has to be chosen with minimal absolute value such
that the inequality in \eqref{eq:optimization-refinement-simplified} is
satisfied as an equality.
\subsubsection*{Coarsening}
For a solution $u^f$ on the fine grid, an optimization problem for the solution
$u^c$ on the coarse grid is
\begin{equation}
\label{eq:optimization-coarsening}
  \min_{u^c} \frac{1}{2} \norm{\Ifc u^c - u^f}_{M^f}^2
  \st 1^T M^c u^c = 1^T M^f u^f,\; 1^T M^c U(u^c) \leq 1^T M^f U(u^f).
\end{equation}
This is again a convex optimization problem possessing a unique solution.

\subsubsection*{Numerical Examples}
Consider the interval $[-1, 1]$ with Lobatto-Legendre nodes for polynomials of
degree $\leq 6$ on the fine grid and $\leq 3$ on the coarse grid. The interpolation
operators are given by polynomial interpolation and the mass matrices as diagonal
matrices using the weights of the corresponding quadrature rule. The exponential
entropy $U(u) = \exp(u)$ is used in the following.
For the refinement experiment in Figure~\ref{fig:p-refinement}, the solution on
the coarse grid is obtained by evaluating $\exp(-x^2)$ on the coarse grid. While
the standard interpolation produces spurious entropy, both optimization approaches
satisfy the entropy inequality up to the tolerances specified for the algorithms.
The optimization results are visually indistinguishable and slightly dissipated
compared to the interpolant.
The coarsening experiment presented in Figure~\ref{fig:p-coarsening} is initialized 
by evaluating $\exp(x)$ on the fine grid. Both the interpolation and standard $L^2$
projection produce spurious amounts of entropy. In contrast, the result obtained by
optimization \eqref{eq:optimization-coarsening} satisfies the desired entropy
inequality.
\begin{figure}
\centering
  \begin{subfigure}[t]{0.47\textwidth}
    \includegraphics[width=\textwidth]{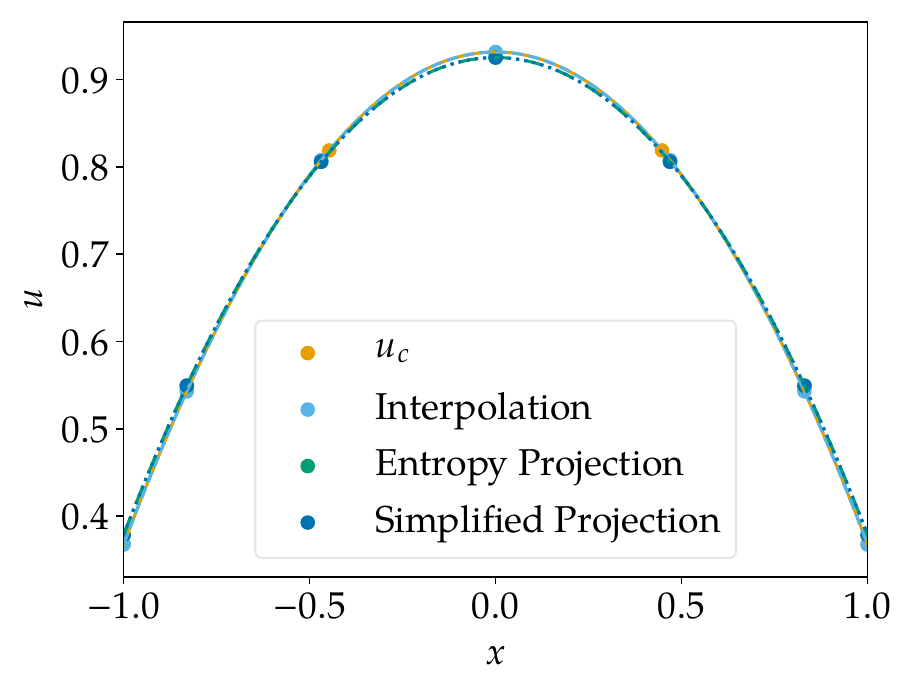}
    \caption{Refinement. The spurious entropy production of the interpolation is
             $3.566\,10^{-3}$, the  approach \eqref{eq:optimization-refinement}
             yields a slight entropy decay of $-3.980\,10^{-8}$ and the simplified
             problem \eqref{eq:optimization-refinement-simplified} results in
             $2.497\,10^{-8}$. Both optimization results deviate from the interpolation
             by $7.562\,10^{-3}$.}
    \label{fig:p-refinement}
  \end{subfigure}%
  \hspace{\fill}
  \begin{subfigure}[t]{0.47\textwidth}
    \includegraphics[width=\textwidth]{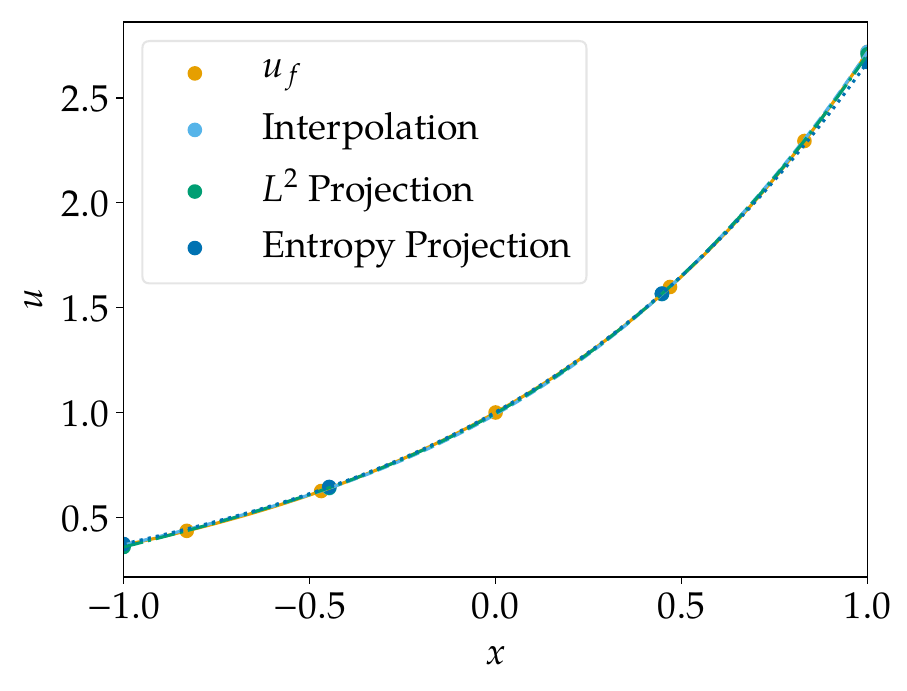}
    \caption{Coarsening. The spurious entropy production of the interpolation is
             $9.749\,10^{-2}$, the amount of entropy produced by the $L^2$ projection
             is $7.767\,10^{-2}$, and the optimization approach results in a slight
             entropy decay of $-3.429\,10^{-7}$.}
    \label{fig:p-coarsening}
  \end{subfigure}%
  \caption{$p$ refinement and coarsening using polynomials of degree $\leq 6$.}
\end{figure}

\subsection{Numerical Simulations - Linear Advection}\label{eq:advection}
In the following example, it is demonstrated that an energy/entropy
(in-)equality does not imply a good numerical approximation
and that the quality of the solution highly depends on the baseline method. 
One considers the linear advection equation 
\begin{equation*}
\begin{aligned}
  \partial_t u(t,x) + \partial_x u(t,x) &= 0, && x \in [-1,1], t \in (0,4), \\
  u(0,x) &= u_0(x) = \sin(\pi x), && x \in [-1,1],
\end{aligned}
\end{equation*}
with periodic boundary conditions. To solve this problem  a nodal DG scheme 
 with $N=16$ uniform elements with polynomial of degree $p=4$ 
 is used where for the calculation of the mass matrix (and then entropy correction term)
 the closed Newton-Cotes quadrature rule 
 (i.e. $p+1$ equidistant nodes including the
boundaries and maximal quadrature accuracy for these nodes, also known as
Boole's rule for $p=4$ \cite[Eq.~25.4.14]{abramowitz1972handbook}) is applied. 
It should be stressed that using equidistant nodes should actually avoided except 
one follows the approach  presented in \cite{glaubitz2020stable}.

The $L^2$ entropy/energy $U(u) = \frac{1}{2} u^2$ is finally considered 
and the spatial semidiscretization is integrated in time using SSPRK(10,4) of
\cite{ketcheson2008highly}, which is energy stable for linear semidiscretizations
\cite{ranocha2018L2stability}.
\begin{figure}
\centering
  \begin{subfigure}{0.45\textwidth}
    \centering
    \includegraphics[width=\textwidth]{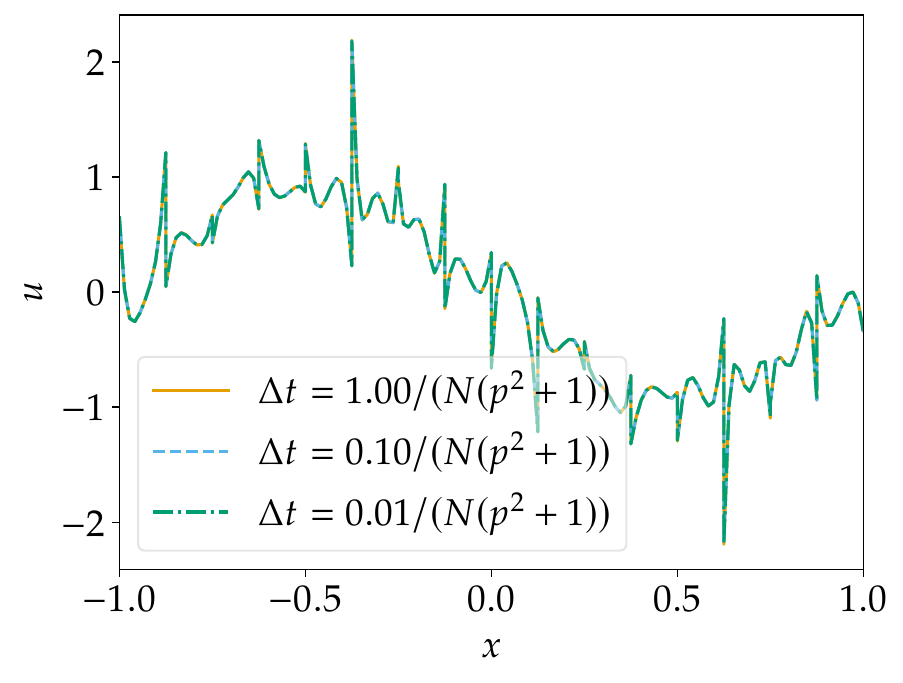}
    \caption{Numerical solutions.}
  \end{subfigure}%
  \hspace{\fill}
  \begin{subfigure}{0.45\textwidth}
    \centering
    \includegraphics[width=\textwidth]{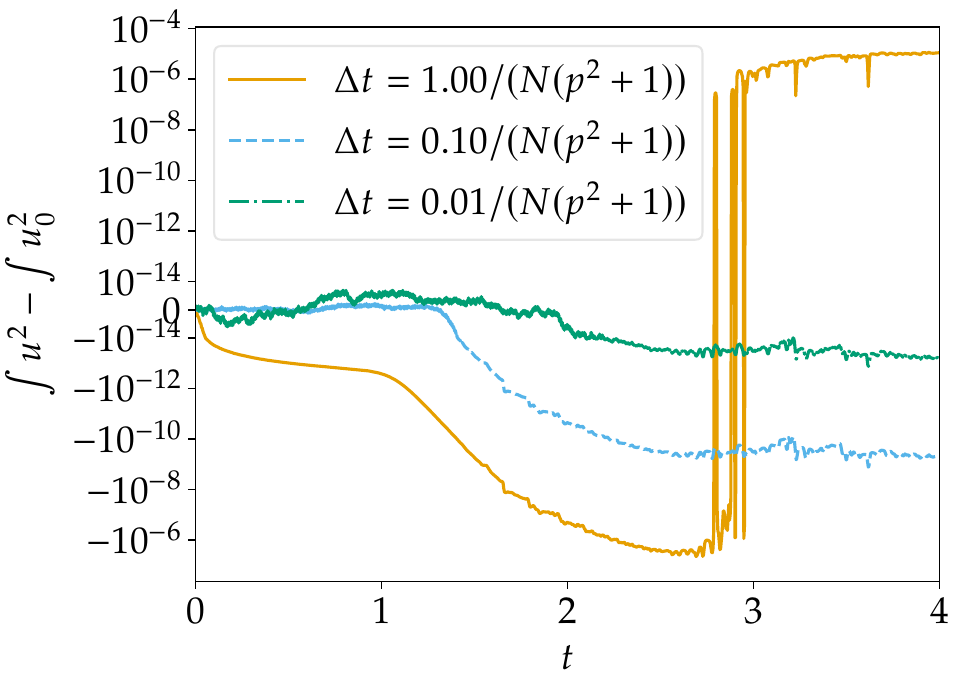}
    \caption{Entropy/energy.}
  \end{subfigure}%
  \caption{Numerical solutions of the linear advection equation and their
           energy/entropy, computed using a nodal DG scheme based on closed
           Newton Cotes quadrature.}
  \label{fig:linear_advection_NewtonCotes_DG}
\end{figure}
Results of these simulations are visualized in
Figure~\ref{fig:linear_advection_NewtonCotes_DG}. Since the semidiscretization
is not linear because of the correction terms, the fully discrete scheme does
not satisfy an entropy/energy inequality. However, the entropy/energy becomes
constant to machine accuracy if the time step $\Delta t$ is refined. Note that the decrease of entropy change corresponds exactly to the order of the time integration method, i.e.\ decreasing the time step by a factor of ten reduces
the entropy change by a factor of $10^4$. Nevertheless,
the numerical solutions are highly oscillatory. While the entropy correction reduces
the amount of oscillations compared to the numerical solution without correction
(not shown), the basically bad behavior of the nodal DG schemes with closed
Newton Cotes quadrature can still be observed.

\subsection{Numerical Simulations - SBP-SAT-FD and Correction terms}\label{subsect:SBP_SAT_FD}
We consider the compressible Euler equations together with Taylor-Green vortex
initial conditions as in Section~\ref{subsection_SBP_DG}.
Using classical  sixth order SBP central finite difference operators \cite{ranocha2018thesis} with $100$ nodes per coordinate
direction, the numerical solutions have been computed in the time interval
$t \in [0,30]$ with the fourth order, ten-stage, strong stability preserving
Runge-Kutta method of \cite{ketcheson2008highly} and a constant time step
$\Delta t = \Delta x / 8$.

\begin{figure}
\centering
 \includegraphics[width=\textwidth]{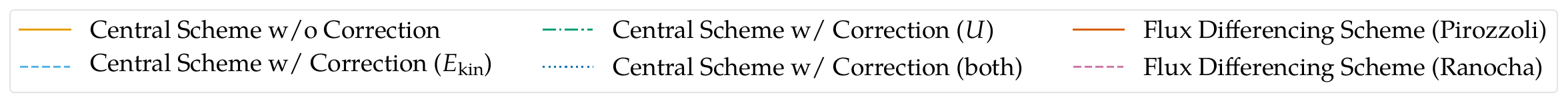}
 \\
 \begin{subfigure}{0.49\textwidth}
   \centering
   \includegraphics[width=\textwidth]{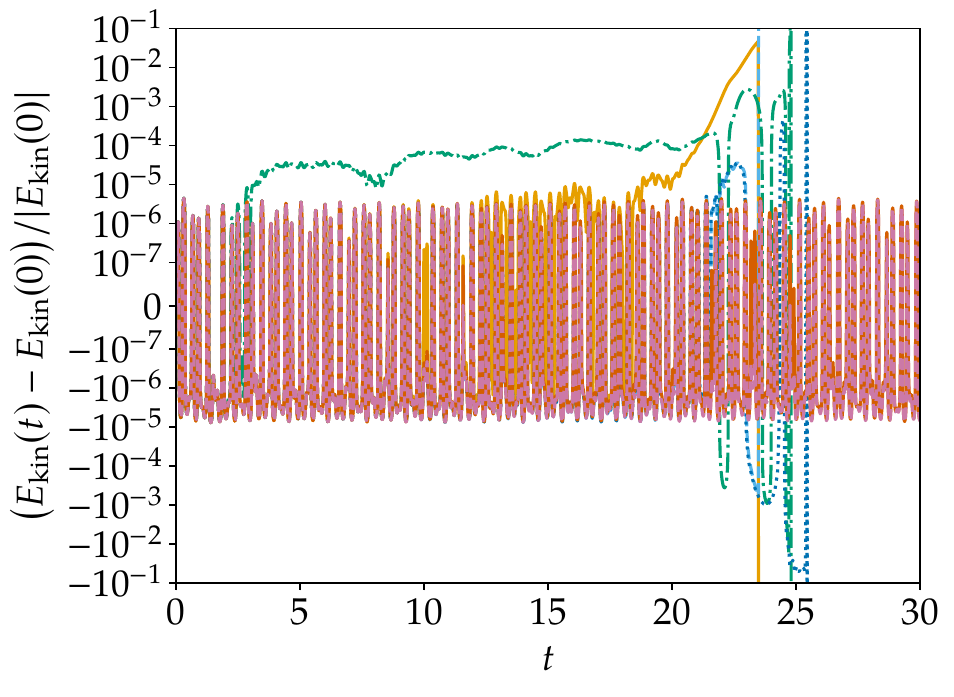}
   \caption{Kinetic Energy.}
 \end{subfigure}%
 \begin{subfigure}{0.49\textwidth}
   \centering
   \includegraphics[width=\textwidth]{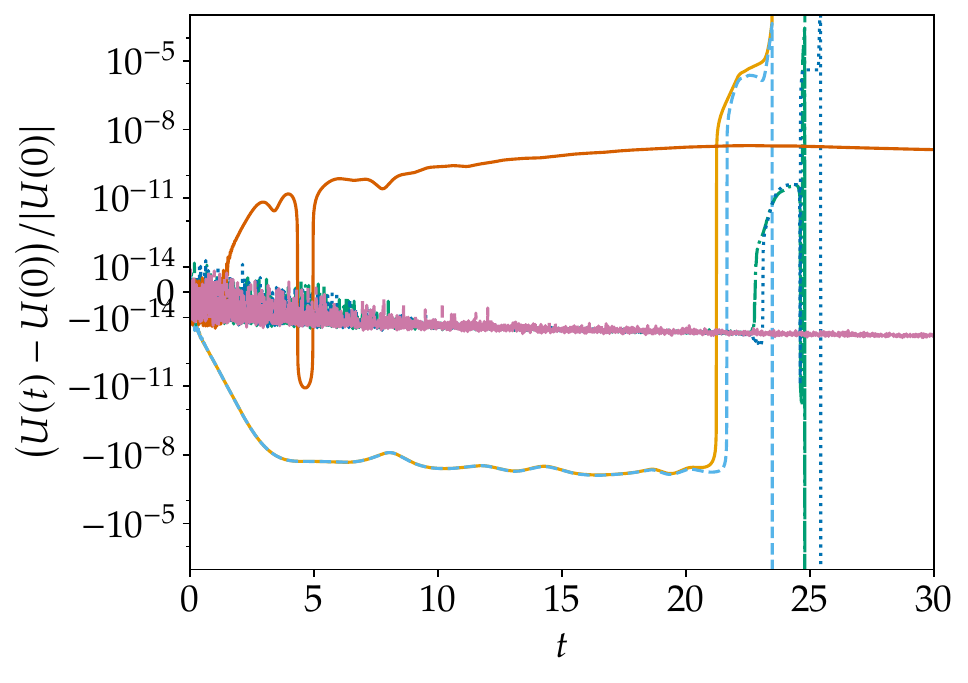}
   \caption{Entropy.}
 \end{subfigure}%
 \caption{Relative kinetic energy $\Ekin$ and entropy $U$
          of numerical solutions of the compressible Euler
          equations with a Taylor-Green vortex initial condition. The finite
          difference schemes use sixth order classical central stencils and
          either the classical central scheme with or without correction term
          or flux difference schemes with numerical fluxes of \cite{pirozzoli2011numerical}
          (see also \cite{gassner2016split}) or
          \cite[Theorem~7.8]{ranocha2018thesis}.}
 \label{fig:Euler_FD_TGV}
\end{figure}

The relative kinetic energy and entropy of numerical solutions obtained via
the classical central scheme with or without corrections or flux difference
schemes are visualized in Figure~\ref{fig:Euler_FD_TGV}.
The classical central scheme blows up at $t \approx 23.50$ and strong variations
of both the kinetic energy and the entropy can be observed shortly before.
Using a correction term for the kinetic energy reduces the variations of $\Ekin$
before the blow-up slightly and does not significantly influence the entropy
before $t \approx 20$. However, the scheme blows up at approximately the same time.

Using instead a correction term for the entropy, the scheme crashes a bit later
at $t \approx 24.82$. This correction term results in an increase of the kinetic
energy in this case. The entropy remains constant until it starts to vary before
the blow-up.
Applying the combined correction term for both entropy and kinetic energy yields
to some extent combined results: The oscillations of the kinetic energy are very
similar to those with the $\Ekin$ correction and the entropy develops similarly
to the one using only the $U$ correction. However, the scheme using the combined
correction term crashes a bit later ($t \approx 25.46$) than the other ones.

Here, the usage of the correction term can improve the performance of
the baseline scheme but can not rescue it. Slightly before the crash,
negative pressures and density can be observed. Hence, the correction term
cannot be applied and the simulation crashes. The applications of limiters
may resolve this problem and will be part of future research in this direction.
Nevertheless, substituting the central scheme with a non-trivial flux difference
discretization improves the numerical stability significantly if appropriate
numerical fluxes are used. In the following, the numerical fluxes of Pirozzoli
\cite{pirozzoli2011numerical} and Ranocha \cite[Theorem~7.8]{ranocha2018thesis}
are used. Both are kinetic energy preserving, i.e.
they satisfy \eqref{eq:operator-Ekin-target} analytically.
The last flux is also entropy
conservative, i.e. the corresponding scheme satisfies \eqref{eq:operator-entropy-target}.
The flux difference scheme does not crash during the computation if either one of
these fluxes is used. However, it should not be hidden that
even using these fluxes the test crashes around $t\approx 35$ but up to this point
the flux splitting approach using the flux \cite[Theorem~7.8]{ranocha2018thesis}
yield the best results up to our knowledge (without any additional techniques).
Both show some oscillations of the kinetic energy with an
amplitude similar to the central scheme with $\Ekin$ correction.
The entropy varies in time if Pirozzoli's flux is used and the amount of variation
is similar to the one for the central scheme before the blow-up. Remarkably, the
scheme using the flux of \cite[Theorem~7.8]{ranocha2018thesis} and the central
schemes with entropy correction conserve the entropy nearly to machine accuracy.
This can not necessarily be expected, since the scheme is only semidiscretely entropy
conservative and the time integration scheme can (and will typically) cause
variations of $U$.
As it can be seen, the flux-splitting with the used fluxes
works  better than a central scheme with correction. However,
it should be pointed out  that quite a lot of efforts have been made
to construct these types of fluxes and several structural properties
have to be assumed to use them. Instead the correction terms are  universal tools
and do not need these properties.
Further, if we apply the correction terms to a skew-symmetric or split
formulation of the volume term which does not need necessarily to be
energy preserving or entropy conservative then similar results
are obtain as for the flux splitting scheme using Ranocha's flux as seen before.

\section*{Acknowledgements}

P\"O has been funded by the SNF project (Number 175784), the UZH Postdoc Scholarship  (Number FK-19-104) and the Gutenberg Research Fellowship. 
The third author was supported by the German Research Foundation (DFG, Deutsche
Forschungsgemeinschaft) under Grant SO~363/14-1.
Funded by the Deutsche Forschungsgemeinschaft (DFG, German Research Foundation)
under Germany's Excellence Strategy EXC 2044-390685587, Mathematics Münster:
Dynamics-Geometry-Structure.
Research reported in this publication was supported by the
King Abdullah University of Science and Technology (KAUST).

\printbibliography

\end{document}